\documentclass[11pt,a4]{amsart}
\usepackage{amsfonts,amsmath,amssymb,bbm}
\usepackage{verbatim}

\makeatletter
\renewcommand{\section}{\@startsection{section}{1}{0pt}{20pt}{6pt}{\large\bfseries}}
\makeatother

\usepackage{enumerate}
\usepackage{pdfsync}
\usepackage{color}

\numberwithin{equation}{section}

\theoremstyle{plain}
  \newtheorem{theorem}{Theorem}[section]
  \newtheorem{proposition}[theorem]{Proposition}
  \newtheorem{lemma}[theorem]{Lemma}
   \newtheorem{corollary}[theorem]{Corollary}
\newtheorem{definition}[theorem]{Definition}
  \newtheorem{remark}[theorem]{Remark}

\def\QED{\hfill $\Box$}
\def \\ { \cr }
\def\R{\mathbb{R}}
\def \1{1 \mkern -6mu 1}

\def\E{\mathbb{E}}
\def\P{\mathbb{P}}

\def \e{{\rm e}}
\def \d{{\rm d}}

\newcommand{\LL}{L\'{e}vy }
\newcommand{\LLP}{L\'{e}vy process }
\newcommand{\LLPs}{L\'{e}vy processes }
\newcommand{\PP}{\overline{\Pi}}
\newcommand{\PPp}{\overline{\Pi}_{+}}

\newcommand{\PPP}{\overline{\overline{\Pi}}}
\newcommand{\PPPn}{\overline{\overline{\Pi}}_{-}}
\newcommand{\PPPp}{\overline{\overline{\Pi}}_{+}}
\newcommand{\IInf}{\int_{0}^{\infty}}

\newcommand{\lb}{\left(}

\newcommand{\rb}{\right)}
\begin{document}
\title[Exponential Functional]{A Wiener-Hopf type factorization for the exponential functional of L\'evy processes}

\author{J.C. Pardo}
\address{Centro de Investigaci\'on en Matem\'aticas A.C. Calle Jalisco s/n. 36240 Guanajuato, M\'exico.}
\email{jcpardo@cimat.mx}

\author{P. Patie}

\address{D\'epartment de math\'ematique, Universit\'e Libre de Bruxelles, B-1050 Bruxelles, Belgium.}
\email{ppatie@ulb.ac.be}

\author{M. Savov}
\address{ New College, University of Oxford, Oxford, Holywell street OX1 3BN, UK.}
\email{mladensavov@hotmail.com} \email{savov@stats.ox.ac.uk}

\thanks{The authors are grateful to the reviewer's valuable comments that improved the presentation of the manuscript. The two last authors would like to thank  A.E. Kyprianou for stimulating  discussions while they were visiting the University of Bath.}
\begin{abstract}
For a  L\'evy process $\xi=(\xi_t)_{t\geq0}$  drifting to $-\infty$, we define the so-called exponential functional as follows
\[{\rm{I}}_{\xi}=\int_0^{\infty}e^{\xi_t} dt.\]
Under mild conditions on $\xi$, we show that  the following factorization of exponential functionals
\[{\rm{I}}_{\xi}\stackrel{d}={\rm{I}}_{H^-} \times {\rm{I}}_{Y}\]
holds, where, $\times $ stands for the product of  independent random variables, $H^-$ is  the descending ladder height process of $\xi$ and $Y$ is a spectrally positive L\'evy process with a negative mean constructed from its ascending ladder height process. As a by-product, we  generate an integral or power series representation for the law of   ${\rm{I}}_{\xi}$ for a large class of L\'evy processes with two-sided jumps and also  derive some new distributional properties. The proof of our main result relies on a fine Markovian study of  a class of generalized Ornstein-Uhlenbeck processes which is of independent interest on its own. We use and refine an alternative approach of studying the stationary measure of a Markov process which avoids some technicalities and  difficulties that appear in the classical method of employing the generator of the dual Markov process.
\newline

\end{abstract}

\maketitle
\newpage
\section{Introduction and main results}
 We are interested in  studying the law of  the so-called exponential functional of L\'evy processes which is defined as follows
 \[{\rm{I}}_{\xi}=\int_0^{\infty}e^{\xi_t} dt,\]
where $\xi=(\xi_t)_{t\geq0}$ is a \LLP starting from $0$ and drifting to $-\infty$.
Recall that a \LLP $\xi$ is a process with stationary and independent increments and its law is characterized completely by its L\'evy-Khintchine exponent $\Psi$ which takes the following form
\begin{equation}\label{Levy-K}
\log \E\left[e^{z \xi_{1}}\right]=\Psi(z)=bz +\frac{\sigma^{2}}{2}z^{2}+\int_{-\infty}^{\infty}\left(e^{zy} - 1 - zy\mathbb{I}_{\{|y|<1\}}\right)\Pi(dy), \text{ for any $z\in i\R$,}
\end{equation}
where $\sigma\geq0, b \in \R$ and $\Pi$ is a \LL measure  satisfying the condition
$\int_{\mathbb{R}}(y^{2}\wedge 1)\Pi(dy)<\infty$. See \cite{Bertoin-96} for more information on \LL processes.

The exponential functional ${\rm{I}}_{\xi}$ has attracted the interest of many researchers over the last two decades. This is mostly due to the prominent role played by the law of ${\rm{I}}_{\xi}$ in the study of important processes, such as self-similar Markov processes, fragmentation and branching processes but also in various settings ranging from astrophysics, biology  to  financial and insurance mathematics, see the survey paper \cite{Bertoin-Yor-05}.

 So far there are two main approaches which have been developed and used to derive information about the law of the exponential functional. The first one uses the fact that the Mellin transform of ${\rm{I}}_{\xi}$ is a solution to a functional equation, see \eqref{Maulik} below, and is due to Carmona et al.~\cite{Carmona-Petit-Yor-97} and has been extended by Maulik and Zwart \cite{Maulik-Zwart-06}. It is important to note that $\eqref{Maulik}$ is useful only under the additional assumption that $\xi$ possesses some finite, positive exponential moments since then it is defined on a strip in the complex plane. This equation can be solved for exponential functionals of negative of subordinators and spectrally positive L\'evy processes yielding some simple expressions for their positive and negative integer moments respectively, which, in both cases, determine the law.  Recently, Kuznetsov and Pardo \cite{Kuznetsov-Pardo-11} have used some special instances of L\'evy processes, for which the solution of the functional equation can directly be guessed and verified from \eqref{Maulik}, to derive some information concerning the law of ${\rm{I}}_{\xi}$. It is worth pointing out that, in general, it is not an easy exercise to invert the Mellin (or moments) transform of $\rm{I}_{\xi}$ since a fine analysis of its asymptotic behavior is required. This  Mellin transform approach relies on two difficult tasks: to find a solution of the functional equation and to provide a general criterion to ensure the uniqueness of its solution. For instance, this approach does not seem to successfully cope with the whole class of spectrally negative L\'evy processes.

 The second methodology, which has been developed  recently by the second author in \cite{Patie-06c} and \cite{Patie-abs-08}, is based on the well-known  relation  between the law of ${\rm{I}}_{\xi}$ and  the distribution of the absorption time of  positive self-similar Markov processes which were introduced by Lamperti \cite{Lamperti-72} in the context of  limit theorems for Markov processes. Indeed, in \cite{Patie-abs-08}, it is  shown that the law of ${\rm{I}}_{\xi}$ can be expressed as an invariant function of a transient Ornstein-Uhlenbeck companion process to the self-similar Markov process. Using some  potential theoretical devices, a power series and  a contour integral representation of the density is provided   when $\xi$ is a possibly killed spectrally negative L\'evy process.

 In this paper, starting from a large class of L\'evy processes, we show that the  law  of ${\rm{I}}_{\xi}$ can be factorized into the product of independent exponential functionals associated with two companion L\'evy processes, namely the descending ladder height process of $\xi$ and  a spectrally positive L\'evy process constructed from its ascending ladder height process. It is well-known that these two subordinators appear in the Wiener-Hopf factorization of L\'evy processes. The laws of these exponential functionals are uniquely determined either by their positive or negative integer moments. Moreover, whenever the law of any of these can be expanded in series we can in general develop the law of ${\rm{I}}_{\xi}$ in series. Thus, for example, the requirements put on the \LL measure of $\xi$ in \cite{Kuznetsov-Pardo-11} can be relaxed to conditions only on the positive jumps (the \LL measure on the positive half-line) of $\xi$ thus enlarging considerably the class of \LLPs $\xi$, for which we can obtain a series expansion of the law of ${\rm{I}}_{\xi}$.

  Although our main result may have a formal explanation through the Wiener-Hopf factorization combined with the functional equation \eqref{Maulik}, the proof is rather complicated and involves a careful study of some generalized Ornstein-Uhlenbeck (for short GOU) processes, different from the ones  mentioned above. For this purpose, we deepen  a technique used by Carmona et al.~\cite[Proposition 2.1]{Carmona-Petit-Yor-97}
  and further developed in \cite{Kuznetsov-Pardo-Savov-11}, which relates the law  of ${\rm{I}}_{\xi}$ to  the stationary measure of a GOU process. More precisely, we show that the density function of ${\rm{I}}_{\xi}$, say $m_{\xi}$, is, under very mild conditions, the unique function satisfying the equation $\mathcal{L}m_{\xi}=0$, where $\mathcal{L}$ is an "integrated infinitesimal" operator, which is strictly of an integral form. The latter allows for a smooth and effortless application of Mellin and Fourier transforms. We believe this method itself will attract some attention as it removes generic difficulties related to the study of the invariant measure via the dual Markov process such as the lack of smoothness properties for the density of the stationary measure and also application of transforms which usually requires the use of Fubini  Theorem which is difficult to verify when dealing with non-local operators.

 Before stating our main result let us introduce some notation. First, since in our setting  $\xi$ drifts to $-\infty$, it is well-known that the ascending (resp.~descending) ladder height process $H^+=(H^+(t))_{t\geq0}$  (resp.~$H^{-}=-H^{-,*}=(-H^{-,*}(t))_{t\geq0}$) is a  killed (resp.~proper) subordinator.   Then,  we write,  for any $z\in i\mathbb{R}$,
 \begin{equation}\label{L-KLadder}
 \phi_{+}(z)=\log \E\left[\exp(zH^+(1))\right] =
 \delta_{+}z + \int_{(0,\infty)}(\e^{zy}-1)\mu_+(\d y)-k_{+}\,,
 \end{equation}
 where $\delta_+\geq0$ is the drift and $k_{+}>0$ is the killing rate. Similarly,  with $\delta_-\geq0$,  we have
  \begin{equation}\label{L-KLadder1}
 \phi_{-}(z)=\log \E\left[\exp(zH^-(1))\right]=
 -\delta_{-}z -\int_{(0,\infty)}(1-\e^{-zy})\mu_-(\d y)\,.
 \end{equation}
We recall that the integrability condition $\int_0^{\infty} (1\wedge y)\mu_{\pm}(dy)<\infty$ holds. The Wiener-Hopf factorization then reads off as follows
  \begin{equation}\label{eq:wh}
 \Psi(z)=-c\phi_{+}(z)\phi_{-}(z)=-\phi_{+}(z)\phi_{-}(z), \text{ for any $z\in i\R$,}
 \end{equation}
where we have used  the convention that the local times  have been normalized in a way that $c=1$, see (5.3.1) in \cite{Doney}. 
 We avoid further discussion as we assume \eqref{eq:wh} holds with $c=1$.

\begin{definition}\label{Definition}
 We denote by $\mathcal{ P}$   the set of positive measures on $\R_{+}$ which admit a non-increasing density.
\end{definition}
Before we formulate the main result of our paper we introduce the two main hypothesis:
\begin{enumerate}
\item[($\mathcal{H}_1$)]   Assume further that $-\infty<\E\left[\xi_{1}\right]$ and that one of the following conditions holds:
\begin{description}
 \item[E${}_+$] $\mu_+ \in \mathcal{ P}$ and there exists $z_+>0$ such that for all $z$ with, $\Re(z) \in (0,z_+),$ we have $|\Psi(z)|<\infty$.
\item [P+] $\Pi_+ \in \mathcal{ P}$.
\end{description}
\item [($\mathcal{H}_2$)] Assume that
\begin{description}
\item[P${}_{\pm}$]  $\mu_+ \in \mathcal{ P},\: k_+>0$ and $\mu_-  \in \mathcal{ P}.$
\end{description}
\end{enumerate}
Then the following result holds.
\begin{theorem}\label{MainTheorem}
Assume that  $\xi$ is a L\'evy process that drifts to $-\infty$ with characteristics of the ladder height processes as in \eqref{L-KLadder} and \eqref{L-KLadder1}. Let either ($\mathcal{H}_1$) or ($\mathcal{H}_2$) holds. Then,
in both cases,  there exists a spectrally positive L\'evy process $Y$ with a negative mean whose Laplace exponent $\psi_+$ takes the form  \begin{equation}
\label{Phi}\psi_+(-s)=-s\phi_{+}(-s)=\delta_{+}s^{2}+k_{+}s+s^{2}\int_{0}^{\infty}e^{-sy}\mu_+(y,\infty)dy,\: s\geq0,
\end{equation}
and the following factorization holds 
\begin{equation}\label{MainAssertion}
{\rm{I}}_{\xi}\stackrel{d}={\rm{I}}_{H^{-}}\times {\rm{I}}_{Y}
\end{equation}
where $\stackrel{d}=$ stands for the  identity in law and  $\times$  for the product of independent random variables. 
\end{theorem}
\begin{remark}
 We  mention that the case when the mean is $-\infty$ together with other problems will be treated in a subsequent study as it demands techniques different from the spirit of this paper.
\end{remark}
The result in Theorem \ref{MainTheorem} can be looked at from another perspective. Let us have two subordinators with \LL measures $\mu_{\pm}$ such that $\mu_+ \in \mathcal{ P},\: k_+>0$ and $\mu_-  \in \mathcal{ P}$. Then according to Vigon's theory of philanthropy, see \cite{Vigon}, we can construct a process $\xi$ such that its ladder height processes have exponents as in \eqref{L-KLadder} and \eqref{L-KLadder1} and hence $\xi$ satisfies the conditions of Theorem \ref{MainTheorem}. Therefore we will be able to synthesize examples starting from the building blocks, i.e. the ladder height processes. We state this as a separate result.
\begin{corollary}\label{CorollaryMain}
Let $\mu_{\pm}$ be the \LL measures of two subordinators and $\mu_+ \in \mathcal{ P},\: k_+>0$ and $\mu_-  \in \mathcal{ P}$. Then there exists a \LLP which drifts to $-\infty$ whose ascending and descending ladder height processes have the Laplace exponents respectively \eqref{L-KLadder} and \eqref{L-KLadder1}. Then all the claims of Theorem \ref{MainTheorem} hold and in particular we have the factorization \eqref{MainAssertion}.
\end{corollary}
 We postpone the proof of the Theorem to the Section \ref{proof:mt}. In the next section, we provide some interesting consequences  whose proofs will be given in Section \ref{proof_cons}.
Finally, in Section \ref{O-U}, we state and prove several  results concerning some generalized Ornstein-Uhlenbeck processes. They will be useful for our main proof and since they have an independent interest, we present them in a separate section.

\section{Some consequences of Theorem \ref{MainTheorem}}\label{SecMain}
Theorem \ref{MainTheorem} allows for a multiple of applications. In this section we discuss only a small part of them but we wish to note that almost all results that have been obtained in the literature under restrictions on all jumps of $\xi$ can now be strengthened by imposing conditions only on positive jumps.
This is due to \eqref{MainAssertion} and the fact that on the right-hand side of the identity the law of the exponential functionals has been determined by its integral moments which admit some simple expressions, see Propositions \ref{prop:ms} and \ref{prop:msp}  below.

The factorization allows us to derive some interesting distributional properties. For instance, we can show that  the random variable ${\rm{I}}_{\xi}$ is unimodal  for a large class of L\'evy processes. We recall that a positive random variable (or its distribution function) is said to be unimodal if  there exists $a \in \mathbb{R}^+$, the mode, such that its distribution function $F(x)$ and the function $1-F(x)$ are convex respectively on $(0, a)$ and $(a,+\infty)$. It can be easily shown, see e.g.~\cite{Rivero-05}, that the random variable ${\rm{I}}_{Y}$, as defined in Theorem \ref{MainTheorem}, is self-decomposable and thus, in particular, unimodal. It is natural to ask whether this property is preserved or not for ${\rm{I}}_{\xi}$. We emphasize that this is not necessarily true even if ${\rm{I}}_{H^-}$ is unimodal itself. Cuculescu and Theodorescu \cite{Cuculescu} provide a criterion for a positive random variable to be multiplicative strongly unimodal (for short MSU), that is, its product with any independent unimodal random variable remains unimodal. More precisely, they show that either the random variable has a unique mode at $0$ and the independent product with any random variable has also an unique mode at $0$ or the law of the positive random variable is absolutely continuous with a density $m$ having the property that the mapping $x\rightarrow \log m(e^x)$ is concave on $\R$. We also point out that it is easily seen that the MSU
property remains unchanged under rescaling and power transformations and we refer to the recent paper \cite{Simon} for more information about this class of random variables.

We proceed by recalling that as a general result on the exponential functional Bertoin et al.~\cite[Theorem 3.9]{Bertoin-Lindner-07} have shown that the law of ${\rm{I}}_{\xi}$ is absolutely continuous  with a density which we denote throughout by  $m_{\xi}$.

In what follows, we show that when $\xi$ is a spectrally negative L\'evy process (i.e.~$\Pi(dy)\mathbb{I}_{\{y>0\}}\equiv 0$ in \eqref{Levy-K} and $\xi$ is not the negative of a subordinator), we recover the power series representation obtained by the second author in \cite{Patie-abs-08} for the density of ${\rm{I}}_{\xi}$. We are now ready to state the first consequence of our main factorization.
\begin{corollary}\label{Corollary1}
Let $\xi$ be a spectrally negative L\'evy process with a negative mean.
\begin{enumerate}
\item Then, we have the following factorization
\begin{equation}\label{SpecNeg}
{\rm{I}}_{\xi}\stackrel{d}={\rm{I}}_{H^{-}} \times G^{-1}_{\gamma},
\end{equation}
where $G_{\gamma}$ is a Gamma random variable of parameter $\gamma>0$, where $\gamma>0$ satisfies the relation $\Psi(\gamma)=0$. Consequently, if ${\rm{I}}_{H^-}$ is unimodal then ${\rm{I}}_{\xi}$ is unimodal.
\item The density function of ${\rm{I}}_{\xi}$ has the form
\begin{equation}\label{SpecNeg1}
m_{\xi}(x)=\frac{x^{-\gamma-1}}{\Gamma(\gamma)}\int_{0}^{\infty}e^{-y/x}y^{\gamma}m_{H^-}(y)dy, \: x>0,
\end{equation}
where $\Gamma$ stands for the Gamma function. In particular, we have
\[\lim_{x\rightarrow \infty }x^{\gamma+1}m_{\xi}(x) = \frac{\E[{\rm{I}}_{H^-}^{\gamma}]}{\Gamma(\gamma)}. \]
\item  Moreover, for any $1/x<\lim_{s\rightarrow \infty}\frac{\Psi(s)}{s}$,
 \begin{eqnarray}
m_{{\xi}}(x)
&=&\frac{\E[{\rm{I}}_{H^-}^{\gamma}]}{\Gamma(\gamma)\Gamma(\gamma+1)}x^{-\gamma-1}\sum_{n=0}^{\infty}(-1)^n \frac{\Gamma(n+\gamma+1)}{\prod_{k=1}^{n}\Psi(k+\gamma)}x^{-n}.
\end{eqnarray}
\item Finally, for any $\beta\geq\gamma+1$, the mapping $x\mapsto x^{-\beta}m_{\xi}(x^{-1})$ is completely monotone on $\R^+$, and, consequently, the law of the random variable ${\rm{I}}^{-1}_{\xi}$ is infinitely divisible with a decreasing density whenever $\gamma \leq 1$.
\end{enumerate}
\end{corollary}
\begin{remark}
\begin{enumerate}
\item From \cite[Corollary VII.5]{Bertoin-96} we get that
\begin{equation*} 
\lim_{s\rightarrow \infty} \frac{\Psi(s)}{s}= \begin{cases}
b-\int_{-1}^0 y \Pi(dy) & \textrm{ if } \sigma=0 \textrm{ and } \int_{-\infty}^0 (1 \wedge y)\Pi(dy)<\infty,\\
+\infty & \textrm{ otherwise.}
\end{cases}
\end{equation*}
 Since we excluded the degenerate cases, we easily check that $b-\int_{-1}^0 y\Pi(dy)>0$.
\item We point out that in \cite{Patie-abs-08}, it is proved that   the density extends to a function  of a complex variable which is analytical on the entire complex plane cut along the negative real axis and admits a power series representation for all $x>0$.
\end{enumerate}
\end{remark}
To illustrate the results above, we consider  $\Psi(s)=-(s-\gamma)\phi_{-}(s),s>0$,  with $\gamma>0$, and where for any $\alpha \in (0,1)$,
\begin{eqnarray} \label{eq:dpa}
-\phi_{-}(s)&=& s\frac{\Gamma(\alpha(s-1)+1)}{\Gamma(\alpha s+1)}\\
&=&\int_0^{\infty}(1-e^{-sy})\frac{(1-\alpha)e^{y/\alpha}}{\alpha \Gamma(\alpha+1)(e^{y/\alpha}-1)^{2-\alpha}}dy=\int_0^{\infty}(1-e^{-sy})\pi_{\alpha}(y)dy \nonumber
\end{eqnarray}
is the Laplace exponent of  a subordinator. Observing  that the density $\pi_{\alpha}(y)$ of the L\'evy measure of $\phi_{-}$ is decreasing, we readily check that $\Psi$ is the Laplace exponent of a spectrally negative L\'evy process.   Next, using the identity ${\rm{I}}_{H^-}\stackrel{(d)}{=}{G_1}^{\alpha}$, see e.g.~\cite{Patie-aff},  we get
\[{\rm{I}}_{\xi}\stackrel{(d)}{=}{G_1}^{\alpha} \times G^{-1}_{\gamma}\]
which, after some easy computations, yields, for any $x>0$,
 \begin{eqnarray}
m_{{\xi}}(x)
&=&\frac{x^{-\gamma-1}}{\Gamma(\gamma)\Gamma(\gamma+1)}\sum_{n=0}^{\infty}\Gamma(\alpha( n+\gamma)+1)\frac{(-x)^{-n}}{n!}\\
&=& \frac{ \Gamma(\alpha \gamma +1) x^{-\gamma-1}}{\Gamma(\gamma)\Gamma(\gamma+1)} {}_1F_0((\alpha,\alpha \gamma+1); -x^{-1}),
\end{eqnarray}
where ${}_1F_0$ stands for the so-called Wright hypergeometric function, see e.g.~\cite[Section 12.1]{Braaksma-64}. Finally, since ${G_1}^{\alpha}$ is unimodal, we deduce that ${\rm{I}}_{\xi}$ is unimodal. Actually, we have a stronger result in this case since ${\rm{I}}_{\xi}$ is itself MSU being the product of two independent MSU random variables, showing in particular that the mapping $x\mapsto {}_1F_0((\alpha,\alpha \gamma+1);e^{x})$ is log-concave on $\R$ for any $\alpha \in (0,1)$ and $\gamma>0$.

We now turn to the second application as an illustration of the situation ${\bf{P +}}$ of Theorem \ref{MainTheorem}. We would like to emphasize that in this case in general we do not require the existence of positive exponential moments. We are not aware of general examples that work without such a restriction as \eqref{Maulik} is always crucially used and it is of real help once it is satisfied on a strip.

\begin{corollary}\label{Corollary2}
Let $\xi$ be a \LLP with $-\infty<\E[\xi_{1}]<0$ and $\sigma^{2}>0$. Moreover assume that
\[\Pi(dy)\mathbb{I}_{\{y>0\}}=c \lambda e^{-\lambda y}dy,\]
where $c,\lambda>0$. Then, we have, for any $s>-\lambda$,
\begin{eqnarray*}
 \psi_+(-s)
&=&\delta_+ s^2+k_+s + c_- \frac{s^{2}}{\lambda+s},
\end{eqnarray*}
where $c_-=c/\phi_-(\lambda)$ and $\delta_+>0$. Consequently, the self-decomposable random variable ${\rm{I}}_{Y}$ admits the following factorization
\begin{equation}\label{eq:hy}
{\rm{I}}_{Y}\stackrel{d}=\delta_+ G^{-1}_{\theta_2} \times B^{-1}(\theta_1,\lambda-\theta_1),
\end{equation}
where  $0<\theta_1<\lambda<\theta_2$ are the two positive  roots of the equation  $\psi_+(s)=0$ and $B$ stands for a Beta random variable. Then, assuming that $\theta_2-\theta_1$ is not an integer, we have, for any $1/x<\lim_{s\rightarrow \infty} |\phi_-(s)|$,
\begin{eqnarray*}
m_{\xi}(x)&=&\frac{k_+\Gamma(\lambda+1)x^{-1}}{  \Gamma(\theta_1+1)\Gamma(\theta_2+1)}\left(\sum_{i=1}^2\frac{\E\left[{\rm{I}}_{H^-}^{\theta_i}\right]}{\Gamma(\theta_i+1)}x^{-\theta_i}\mathcal{I}_{\phi_-,i}(\theta_i+1;-x^{-1})\right),
\end{eqnarray*}
where
\begin{eqnarray}
 \mathcal{I}_{\phi_-,i}(\theta_i+1;x) &=&\sum_{n=0}^{\infty}a_{n}(\phi_-,\theta_i)\frac{x^{n}}{n!}
 \end{eqnarray}
and $a_n(\phi_-,\theta_i)= \prod_{\stackrel{j=1}{j\neq i}}^2\frac{\Gamma(\theta_j-\theta_i-n)}{\Gamma(\lambda-\theta_i-n)}\frac{\Gamma(n+\theta_i+1)}{\prod_{k=1}^{n}\phi_-(k+\theta_i)},\: i=1,2$.
\end{corollary}
\begin{remark}
The assumption  $\sigma^{2}>0$, as well as the restriction on $\theta_2-\theta_1$,  have been made in order to avoid dealing with different cases but they can both be easily removed. The latter will affect the series expansion \eqref{Cor.2.4}. The computation is easy but lengthy and we leave it out.
\end{remark}
\begin{remark}
The methodology and results we present here can also be extended to the case when the \LL measure $\Pi(dy)\mathbb{I}_{\{y>0\}}$ is a mixture of exponentials as  in \cite{Gai-Kou-11} and \cite{Kuznetsov-Pardo-11} but we note that here we have no restrictions on the negative jumps whatsoever.
\end{remark}
We now provide an example of Theorem \ref{MainTheorem} in the situation ${\bf{P}_{\pm}}$.
\begin{corollary}\label{Corollary3}
For any $\alpha \in (0,1)$, let us set
\begin{eqnarray}
\Psi(z)= \frac{\alpha z \Gamma(\alpha (-z+1)+1)}{(1-z)\Gamma(-\alpha   z+1)} \phi_+(z), \:  z \in i\R,
\end{eqnarray}
where $\phi_+$ is as in \eqref{L-KLadder} with  $\mu_+ \in \mathcal{P}, k_+>0$.   Then $\Psi$ is the Laplace exponent of a L\'evy process $\xi$ which drifts to $-\infty$. Moreover, the density of ${\rm{I}}_{\xi}$ admits the following representation
 \begin{eqnarray} \label{eq:dis}
m_{{\xi}}(x)
&=&\frac{x^{-1/\alpha}}{\alpha}\int_0^{\infty}g_{\alpha}\left(\left(y/x\right)^{1/\alpha}\right)m_{Y}(y)y^{1/\alpha-1}dy,\: x>0,
\end{eqnarray}
where $g_{\alpha}$ is the density of a positive $\alpha$-stable random variable. Furthermore, if $\lim_{s\to \infty}s^{\alpha-1}\phi_+(-s)=0$, then  for all $x>0$,
 \begin{eqnarray}
m_{{\xi}}(x)
&=&\frac{k_+}{\alpha}\sum_{n=1}^{\infty} \frac{\prod_{k=1}^{n}\phi_+(-k)}{\Gamma(-\alpha n) n!}x^{n}.
\end{eqnarray}
Finally, the  positive random variable ${\rm{I}}_{H^-}$ is MSU  if and only if  $\alpha\leq 1/2$. Hence ${\rm{I}}_{\xi}$ is unimodal for any $\alpha\leq 1/2$.
\end{corollary}
\begin{remark}
The fact that ${\rm{I}}_{H^-}$ is MSU  if and only if  $\alpha\leq 1/2$ is a consequence of the main result of \cite{Simon}.
\end{remark}
\begin{remark}
Note that this is a very special example of the approach of building the \LL process from $\phi_{\pm}$ when $\mu_{\pm}\in \mathcal{P}$. One could construct many examples like this and this allows for interesting applications in mathematical finance and insurance, see e.g.~\cite{Patie-09-cras}.
\end{remark}

 As a specific instance of the previous result, we may consider the case when \[\phi_+(-s)=-\frac{ \Gamma(\alpha' s+1)}{\Gamma(\alpha'(s+1)+1)},\: s\geq0,\] with $\alpha' \in (0,1)$.  We easily obtain from the identity \eqref{eq:msn} below that
\[\E\left[{\rm{I}}_{Y}^{-m}\right] = \frac{\Gamma(\alpha'm+1-\alpha')}{\Gamma(1-\alpha')}, \: m=1,2,\ldots, \]
that is ${{\rm{I}}_{Y}}\stackrel{d}= G^{-\alpha'}_{1-\alpha'}$.  Hence, as the product of independent MSU random variables,  ${\rm{I}}_{\xi}$ is MSU for any $\alpha' \in (0,1)$ and $\alpha\leq 1/2$. Moreover, using the asymptotic behavior of the ratio  of gamma functions given in \eqref{eq:ag} below, we deduce that  for any $\alpha' \in (0,1-\alpha)$ we have
 \begin{eqnarray}
m_{{\xi}}(x)
&=&\frac{1}{\Gamma(1-\alpha')\alpha}\sum_{n=1}^{\infty} \frac{\Gamma(\alpha' n+1)}{\Gamma(-\alpha n) n!}(-1)^{n}x^{n},
\end{eqnarray}
which is valid for any $x>0$.

 We end this section by describing  another interesting factorization of  exponential functionals. Indeed, assuming that $\mu_-\in \mathcal{P}$,  it is shown in \cite[Theorem 1]{Patie-aff} that there exists a spectrally positive L\'evy process $\overline{Y}=(\overline{Y}_t)_{t\geq0}$ with a negative mean and  Laplace exponent given by $\overline{\psi}_{+}(-s)=-s\phi_-(s+1),\: s>0,$ such that the following factorization of the exponential law
\begin{equation} \label{eq:dy}
{\rm{I}}_{H^-}\times {\rm{I}}^{-1}_{\overline{Y}}\stackrel{d}= G_1
\end{equation}
holds.  Hence, combining \eqref{eq:dy} with \eqref{MainAssertion}, we obtain that
\[{\rm{I}}_{\xi}\times  {\rm{I}}^{-1}_{\overline{Y}}\stackrel{d}= G_1 \times {\rm{I}}_{Y}.\]
Consequently, we deduce from \cite[Theorem 51.6]{Sato-99} the following.
\begin{corollary}\label{Corollary33}
If in  one of the settings of Theorem \ref{MainTheorem}, we assume further that  $\mu_-\in \mathcal{P}$, then the density of  the random variable ${\rm{I}}_{\xi}\times {\rm{I}}^{-1}_{\overline{Y}}$, where ${\rm{I}}_{\overline{Y}}$ is taken as defined in \eqref{eq:dy},  is a mixture of exponential distributions and in particular it is infinitely divisible and non-increasing on $\R^+$.
\end{corollary}
Considering as above that ${\rm{I}}_{H^-}\stackrel{(d)}{=}{G^{\alpha}_1}$ in  Corollary \ref{Corollary1} and \ref{Corollary2}, we deduce from \cite[Section 3.2]{Patie-aff} that
the random variable $S^{-\alpha}_{\alpha}\times {\rm{I}}_{\xi}$ is a mixture of exponential distributions, where $S_{\alpha}$ is  a positive stable law of index $\alpha$.

\section{Some results on  generalized Ornstein-Uhlenbeck processes}\label{O-U}
The results we present here will be central  in the development of the proof of our main theorem. However, they also have some interesting implications in the study of generalized Ornstein-Uhlenbeck processes (for short GOU), and for this reason we state and prove them in a separate section.

 We recall that for a given  L\'evy process $\xi$ the GOU process $U^{\xi}$, is defined, for any $t\geq0,\,x\geq0$, by
\begin{equation}\label{Ornstein-U}
U^{\xi}_t(x)=xe^{\xi_{t}}+e^{\xi_{t}}\int_{0}^{t}e^{-\xi_{s}}ds.
\end{equation}
This family of positive strong Markov processes has been intensively studied by Carmona et al.~\cite{Carmona-Petit-Yor-97} and we refer to \cite{Patie-08a} for some more recent studies and references. The connection with our current problem is explained as follows. From the identity in law $(\xi_{t}-\xi_{(t-s)-})_{0\leq s\leq t}=(\xi_{s})_{s\leq t}$, we easily deduce that, for any fixed $t\geq0$,
\begin{equation*}
U^{\xi}_t(x)\stackrel{d}=xe^{\xi_{t}}+\int_{0}^{t}e^{\xi_{s}}ds.
\end{equation*}
Thus,  we have if $\lim_{t\to \infty}\xi_t=-\infty$ a.s., that
\begin{equation*}
U^{\xi}_{\infty}(x)\stackrel{d}={\rm{I}}_{\xi}
\end{equation*}
and hence the law of ${\rm{I}}_{\xi}$ is the unique  stationary measure of $U^{\xi}$, see \cite[Proposition 2.1]{Carmona-Petit-Yor-97}.

In the sequel we use the standard notation $C_b(\mathbb{R})$ (resp.~$C_b(\mathbb{R}_+)$) to denote the  set of bounded and continuous functions on $\mathbb{R}$ (resp.~on $\mathbb{R}_+$). Furthermore, we set $\mathcal{V}'=C^{2}_{b}(\overline{\mathbb{R}})$, where $C^{2}_{b}(\overline{\mathbb{R}})$ is the set of twice continuously differentiable bounded functions which together with its first two derivatives are continuous on  $\overline{\mathbb{R}}=[-\infty,\infty]$.  Then, we recall that, see e.g.~\cite{Carmona-Petit-Yor-97} for the special case when $\xi$ is the sum of a Brownian motion and an independent  \LLP with  bounded variation and finite exponential moments and \cite{Kuznetsov-Pardo-Savov-11} for the general case, the infinitesimal generator $L^{U^{\xi}}$ of $U^{\xi}$ takes the form
\begin{align}\label{InfGen}
&L^{U^{\xi}}f(x)=L^{\xi}f_e(\ln{x})+f'(x), \: x>0,
\end{align}
 whenever $\E[|\xi_{1}|]<\infty$ and $f_e(x)=f(e^{x})\in Dom(L^\xi)$, where  $L^{\xi}$ stands for the infinitesimal generator of the \LLP $\xi$, considered in the sense of It\^o and Neveu (see
\cite[p.~628-630]{Loeve}). Recall in this sense $\mathcal{V}'\subset Dom(L^\xi)$ and hence $\mathcal{V}=\{f:\overline{\mathbb{R}}_+\mapsto\overline{\mathbb{R}} | f_e\in \mathcal{V}'\}\subset Dom(L^{U^{\xi}})$. 

 In what follows we often appeal to the quantities, defined for $x>0$, by
\begin{equation}\label{PiTail}
\PP(x):=\int_{|y|>x}\Pi(dy);\,\,\PP_{\pm}(x):=\int_{y>x}\Pi_{\pm}(dy),
\end{equation}
\begin{equation}\label{DoublePiTail}
\PPP(x):=\int_{y>x}\PP(y)dy;\,\,\PPP_{\pm}(x):=\int_{y>x}\PP_{\pm}(y)dy,
\end{equation}
 where $\Pi_{+}(dy)=\Pi(dy)1_{\{y>0\}}$ and $\Pi_{-}(dy)=\Pi(-dy)1_{\{y>0\}}$.
Note that the quantities in \eqref{DoublePiTail} are finite when  $\E\left[|\xi_{1}|\right]<\infty$.
Moreover, when $\E[\xi_{1}]<\infty$, \eqref{Levy-K} can be rewritten, for all $z\in\mathbb{C}$, where it is well defined, as follows
\begin{equation}\label{psi}
\Psi(z)=\E\left[\xi_{1}\right]z+\frac{\sigma^{2}}{2}z^{2}+z^{2}\int_{0}^{\infty}e^{zy}\PPPp(y)dy+z^{2}\int_{0}^{\infty}e^{-zy}\PPPn(y)dy.
\end{equation}
For the proof of our main theorem we need to study the stationary measure of $U^{\xi}$ and in particular $L^{U^{\xi}}$ in detail. To this end, we introduce the following functional  space
\begin{equation}\label{K}
\nonumber \mathcal{K}=\Big\{f:\overline{\mathbb{R}}_+\mapsto\overline{\mathbb{R}}|\,f_e\in\mathcal{V}';\,\lim_{x\to -\infty}\Big(|f'_e(x)|+|f''_{e}(x)|\Big)=0;\,\int_{\R} \Big(|f'_e(x)|+|f''_e(x)|\Big)dx<\infty\Big\},
\end{equation}
where $f_e(x)=f(e^{x})$.
\begin{proposition}\label{Prop1}
Let $U^{\xi}$ be a GOU process with $\E[|\xi_{1}|]<\infty$. Then $\mathcal{K}\subset Dom(L^{U^{\xi}})$. Moreover, for any $f\in \mathcal{K}$, we have, for all $x>0$,
\begin{equation}\label{InfGen1}
L^{U^{\xi}}f(x)=\frac{g(x)}{x}+\E[\xi_{1}]g(x)+\frac{\sigma^{2}}{2}xg'(x)+\int_{x}^{\infty}g'(y)\PPPp\Big(\ln{\frac{y}{x}}\Big)dy+\int_{0}^{x}g'(y)\PPPn\Big(\ln{\frac{x}{y}}\Big)dy,
\end{equation}
where $g(x)=xf'(x)$. Finally, for any function $h$ such that $\int_{0}^{\infty}(y^{-1}\wedge 1)|h(y)|dy<\infty$ and $f\in\mathcal{K}$ we have
\begin{equation}\label{InfGen2}
(L^{U^{\xi}}f,h)=(g',\mathcal{L}h),
\end{equation}
where $(f_1,f_2)=\int_0^{\infty}f_1(x)f_2(x)dx$ and
\begin{equation}\label{InfGen3}
\mathcal{L}h(x)=\frac{\sigma^{2}}{2}xh(x)+\int_{x}^{\infty}\left(\frac{1}{y}+\E[\xi_{1}]\right) h(y)dy
+\int_{x}^{\infty}\PPPn\Big(\ln{\frac{y}{x}}\Big)h(y)dy+
\int_{0}^{x}\PPPp\Big(\ln{\frac{x}{y}}\Big)h(y)dy.
\end{equation}
\end{proposition}
\begin{remark}  There are certain advantages when using the linear operator $\mathcal{L}$ instead of the generator of the dual GOU. Its integral form allows for minimal conditions on the integrability of $|h|$ and requires no smoothness assumptions on $h$. Moreover, if $h$ is positive, Laplace and Mellin transforms can easily be applied to $\mathcal{L}h(x)$ since the justification of Fubini Theorem is straightforward.
\end{remark}
\begin{proof}
Let $f\in \mathcal{K}$ then by the very definition of $\mathcal{K}$ we have that $f_e\in\mathcal{V}'$ and from \eqref{InfGen} we get that $\mathcal{K}\subset Dom(L^{U^{\xi}})$. Next, \eqref{InfGen1} can be found in \cite{Kuznetsov-Pardo-Savov-11} but can equivalently be recovered from \eqref{InfGen} by simple computations using the expression for $L^{\xi}$, which can be found on  \cite[ p. 24]{Bertoin-96}.
To get \eqref{InfGen2} and \eqref{InfGen3}, we  recall that $g(x)=xf'(x)=f'_e(\ln x)$ and use \eqref{InfGen1} combined with a formal application of the Fubini Theorem to write 
\begin{eqnarray}\label{eqn:Referee}
\nonumber(L^{U^{\xi}} f, h)&=&\IInf \frac{g(y)}{y}h(y)dy+\frac{\sigma^{2}}{2}\IInf y g'(y)h(y)dy+\E[\xi_{1}]\IInf g(y)h(y)dy+\\
\nonumber & & \IInf
\int_{0}^{y}g'(v)\PPPn\big(\ln{\frac{y}{v}}\big)dv h(y)dy+\IInf\int_{y}^{\infty}g'(v)\PPPp\big(\ln{\frac{v}{y}}\big)dvh(y)dy \\
\nonumber &=&
\IInf g'(v)\int_{v}^{\infty}\frac{h(y)}{y}dydv+\E[\xi_{1}]\IInf g'(v)\int_{v}^{\infty}h(y)dydv+\frac{\sigma^{2}}{2}\IInf vg'(v)h(v)dv+\\
\nonumber & & \IInf g'(v)\int_{v}^{\infty}\PPPn\big(\ln{\frac{y}{v}}\big)h(y)dydv+\IInf g'(v)\int_{0}^{v}\PPPp\big(\ln{\frac{v}{y}}\big)h(y)dydv\\ &=&
(g',\mathcal{L}h).
\end{eqnarray}
To justify Fubini Theorem, note that $f\in\mathcal{K}$ implies that $\lim_{x\to 0}g(x)=\lim_{x\to 0}f'_e(\ln x)=0$, $g(x)=\int_{0}^{x}g'(v)dv$ and
\begin{align}\label{PropertiesG}
\nonumber &\int_{0}^{\infty}|g'(v)|dv= \int_{\R}|f''_e(y)|dy\leq C(g)<\infty,\\
&|g(x)|+x|g'(x)|= |f'_e(\ln x)|+|f''_e(\ln x)| \leq C(g)<\infty,
\end{align}
where $C(g)>0$. Note that \eqref{PropertiesG} and the integrability of $(1\wedge y^{-1})|h(y)|$ imply that
\[\IInf \Big|\frac{g(y)}{y}\Big|h(y)dy\leq \IInf \int_{0}^{y}|g'(v)|dv y^{-1}|h(y)|dy\leq C(g) \IInf y^{-1}|h(y)|dy<\infty,\]
and so Fubini Theorem applies to the first term in \eqref{eqn:Referee}. The second term in \eqref{eqn:Referee} remains unchanged whereas for the third one we do the same computation noting that only $y^{-1}$ is not present. From \eqref{PropertiesG} and the fact that $\PPPp(1)+\PPPn(1)<\infty$ since $\E[|\xi_{1}|]<\infty$, we note that for the other two terms, we have with the constant $C(g)>0$ in \eqref{PropertiesG},
\begin{eqnarray*}\label{Bound1}
\int_{0}^{x}|g'(v)|\PPPn\Big(\ln{\frac{x}{v}}\Big)dv&=&\IInf |xe^{-w}g'(xe^{-w})|\PPPn(w)dw \\
&\leq&\PPPn(1)\int_{0}^{\infty}|g'(v)|dv+C(g)\int_{0}^{1}\PPPn(w)dw<\infty\\
 \int_{x}^{\infty}|g'(v)|\PPPp\Big(\ln{\frac{v}{x}}\Big)dv&=&\IInf |xe^{w}g'(xe^{w})|\PPPp(w)dw \\
&\leq& \PPPp(1)\int_{0}^{\infty}|g'(v)|dv+ C(g) \int_{0}^{1}\PPPp(w)dw<\infty.
\end{eqnarray*}
Therefore we can apply Fubini Theorem which  completes  the proof of Proposition \ref{Prop1}.
\end{proof}
 The next result is known and can be found  in \cite{Kuznetsov-Pardo-Savov-11} but we include it and sketch its proof for sake of  completeness and for further discussion.
\begin{theorem}\label{O-UTheorem1}
Let $U^{\xi}$ be a GOU where $-\infty<\E[\xi_{1}]<0$. Then $U^{\xi}$ has a unique stationary distribution which is absolutely continuous with density $m$ and satisfies
\begin{equation}\label{IntegralEqn}
\mathcal{L}m(x)=0 \text{ for a.e. $x>0$}.
\end{equation}
\end{theorem}

\begin{remark}
 Note that due to the discussion in Section \ref{O-U}, $m=m_{\xi}$, i.e it equals the density of the law of ${\rm{I}}_{\xi}$. Therefore all the information we gathered for $m_{\xi}$ in Section \ref{SecMain} is valid here for the density of the stationary measure of $U^{\xi}$, i.e.~$m$.
\end{remark}
\begin{remark}
 Equation \eqref{IntegralEqn} can be very useful. In this instance it is far easier to be studied than an equation coming from the dual process which is standard when stationary distributions are discussed. It does not presuppose any smoothness of $m$ but only its existence. Moreover, as noted above \eqref{IntegralEqn} is amenable to various transforms and difficult issues such as interchanging integrals using Fubini Theorem are effortlessly overcome.
\end{remark}
\begin{remark} It is also interesting to explore other cases when a similar equation to \eqref{IntegralEqn} can be obtained. It seems the approach is fairly general but requires special examples to reveal its full potential. For example, if $L$ is an infinitesimal generator, $\mathcal{N}$ is a differential operator, $\mathcal{L}$ is an integral operator and it is possible for all $f\in C_{0}^{\infty}(\mathbb{R}_{+})$, i.e. infinitely differentiable functions with compact support, and a stationary density $u$ to write
\[(Lf,u)=(\mathcal{N}f,\mathcal{L}u)=0\]
then we can solve the equation in the sense of Schwartz to obtain
\[\tilde{\mathcal{N}}\mathcal{L}u=0,\]
where $\tilde{\mathcal{N}}$ is the dual of $\mathcal{N}$. If we show that necessarily for probability densities $\mathcal{L}u=0$, then we can use $\mathcal{L}$ to study stationarity.
\end{remark}
\begin{proof}
 From \eqref{InfGen2} and the fact that $m$ is the stationary density we get,  for all $g(x)=xf'(x)$, with $f\in C_{0}^{\infty}(\mathbb{R}_{+})\subset \mathcal{K}$,
\[(g',\mathcal{L}m)=0.\]
Then from Schwartz theory of distributions we get $\mathcal{L}m(x)=C\ln{x}+D$ a.e.. Integrating \eqref{InfGen3} and the right-hand side of the latter from $1$ to $z$, multiply the resulting identity by $z^{-1}$, subsequently letting $z\rightarrow\infty$ and using the fact that $m$ is a probability density we can show that necessarily $C=D=0$. The latter requires some efforts but they are mainly technical.
\end{proof}
\begin{theorem}\label{O-UTheorem}
Let $\overline{m}$ be a probability density function such that $\int_{0}^{\infty}\overline{m}(y)y^{-1}dy<\infty$ and \eqref{IntegralEqn} holds for $\overline{m}$ then
\begin{equation}\label{uniqueness}
m(x)=\overline{m}(x)\text{ a.e.,}
\end{equation}
where $m$ is the density of the stationary measure of $U^{\xi}$.
\end{theorem}
\begin{remark} This result is very important in our studies. The fact that we have uniqueness on a large class of probability measures allows us by checking that $\eqref{IntegralEqn}$ holds to pin down the density of the stationary measure of $U^{\xi}$ which is of course the density of ${\rm{I}}_{\xi}$. The requirement that $\int_{0}^{\infty}\overline{m}(y)y^{-1}dy<\infty$ is in fact no restriction whatsoever since the existence of a first negative moment of ${\rm{I}}_{\xi}$ is known from the literature, see \cite{Bertoin-Yor-02-b}.
\end{remark}
\begin{remark} Also it is well known that if $L^{\hat{U}}$ is the generator of the dual Markov process then $L^{\hat{U}}\overline{m}=0$ does not necessarily have a unique solution when $L^{\hat{U}}$ is a non-local operator. Moreover one needs assumptions on the smoothness of $\overline{m}$ so as to apply $L^{\hat{U}}$. Using $\mathcal{L}$ circumvents this problem.
\end{remark}
\begin{proof}
 Let $(P_{t})_{t\geq 0}$ be the semigroup of the GOU $U^{\xi}$, that is, for any $f\in C_b(\overline{\R}_+)$,
\[P_tf(x)= \E\left[f\left(U^{\xi}_t(x)\right)\right],\: x\geq 0,\,t\geq0.\]
 If \eqref{IntegralEqn} holds for some probability density $\overline{m}$ then \eqref{InfGen2} is valid, i.e. for all $f\in \mathcal{K}$,
\[(L^{U^{\xi}}f,\overline{m})=(g',\mathcal{L}\overline{m})=0.\]
 Assume for a moment that
 \begin{equation}\label{Invariant set}
P_s\mathcal{K}\subset \mathcal{K}, \text{ for all $s>0$},
 \end{equation}
 and, there exists a constant $C(f,\xi)>0$ such that, for all $s\leq t$,
 \begin{equation}\label{Bound}
\left|L^{U^{\xi}}P_{s}f(x)\right|\leq C(f,\xi)(x^{-1}\wedge 1).
 \end{equation}
 Then integrating out with respect to $\overline{m}(x)$ the standard equation
\[P_{t}f(x)=f(x)+\int_{0}^{t}L^{U^{\xi}}P_{s}f(x)ds,\]
we get, for all $f\in\mathcal{K}$,
\[\int_{0}^{\infty}P_{t}f(x)\overline{m}(x)dx=\int_{0}^{\infty}f(x)\overline{m}(x)dx.\]
Since $C_{0}^{\infty}(\mathbb{R}_{+})\subset \mathcal{K}$ and $C_{0}^{\infty}(\mathbb{R}_{+})$ is separating for $C_{0}(\mathbb{R}_{+})$, the last identity shows that $\overline{m}$ is a density of a stationary measure. Thus by uniqueness of the stationary measure we conclude \eqref{uniqueness}. Let us prove \eqref{Invariant set} and \eqref{Bound}. For $f\in \mathcal{K}$ write
 \[g_s(x):=P_{s}f(x)=\E\left[ f\left(U^{\xi}_{s}(x)\right)\right]=\E \left[f\left(xe^{\xi_{s}}+\int_{0}^{s}e^{\xi_{v}}dv\right)\right].\]
 Put $\tilde{g}_{s}(x)=g_{s}(e^x)=(g_{s})_e(x)$ .
Note that since $f\in \mathcal{K}$ and $0<e^{x+\xi_{s}}\leq e^{x+\xi_{s}}+\int_{0}^{s}e^{\xi_{v}}dv$ we have the following bound 
\begin{equation}\label{NewBound}\left|e^{x+\xi_{s}}f'\left(e^{x+\xi_{s}}+\int_{0}^{s}e^{\xi_{v}}dv\right)\right|+\left|e^{2(x+\xi_{s})}f''\left(e^{x+\xi_{s}}+\int_{0}^{s}e^{\xi_{v}}dv\right)\right|\leq C(f)\end{equation}
which holds uniformly in $x\in \R$ and $s\geq 0$.
In view of \eqref{NewBound} the dominated convergence theorem gives
\begin{align}\label{New1}
\nonumber &\tilde{g}'_{s}(x)=\E\left[ e^{x+\xi_{s}}f'\left(e^{x+\xi_{s}}+\int_{0}^{s}e^{\xi_{v}}dv\right)\right],\\
\nonumber &\tilde{g}''_{s}(x)=
  \E\left[ e^{x+\xi_{s}}f'\left(e^{x+\xi_{s}}+\int_{0}^{s}e^{\xi_{v}}dv\right)\right]+\E\left[ e^{2(x+\xi_{s})}f''\left(e^{x+\xi_{s}}+\int_{0}^{s}e^{\xi_{v}}dv\right)\right],\\
  &\max\{|\tilde{g}'_{s}(x)|,|\tilde{g}''_{s}(x)|\}\leq C(f).
\end{align}
Clearly then from \eqref{NewBound}, \eqref{New1}, the dominated convergence theorem and the fact that $f\in \mathcal{K}$ which implies the existence of $\lim_{x\to\infty}f''_{e}(x)=b$, we have
\begin{align*}
& \lim_{x\to\infty}\tilde{g}''_{s}(x)=\E\left[ \lim_{x\to\infty}\left(e^{x+\xi_{s}}f'\left(e^{x+\xi_{s}}+\int_{0}^{s}e^{\xi_{v}}dv\right)+ e^{2(x+\xi_{s})}f''\left(e^{x+\xi_{s}}+\int_{0}^{s}e^{\xi_{v}}dv\right)\right)\right]=
b.
\end{align*}
 Similarly, we show that $\lim_{x\to\infty}\tilde{g}'_{s}(x)=\lim_{x\to\infty}f'_e(x)$ and trivially $\lim_{x\to\pm\infty}\tilde{g}_{s}(x)=\lim_{x\to\pm\infty} f_e(x)$.
Finally using \eqref{NewBound}, \eqref{New1}, $f\in\mathcal{K}$, the dominated convergence theorem  and the fact that for all $s>0$ almost surely $\int_{0}^{s}e^{\xi_{v}}dv>0$, we conclude that
\[ \lim_{x\to -\infty}|\tilde{g}'_{s}(x)|+|\tilde{g}''_{s}(x)|\leq
 2\E\left[\lim_{x\to-\infty}\left|e^{x+\xi_{s}}f'\left(e^{x+\xi_{s}}+\int_{0}^{s}e^{\xi_{v}}dv\right)\right|+\left|e^{2(x+\xi_{s})}f''\left( e^{x+\xi_{s}}+\int_{0}^{s}e^{\xi_{v}}dv\right)\right|\right],\]
which together with the limits above confirms that $\tilde{g}_s \in \mathcal{V}'$ and proves that
\[\lim_{x\to -\infty}|\tilde{g}'_{s}(x)|+|\tilde{g}''_{s}(x)|=0.\]
Finally since $f\in \mathcal{K}$ and \eqref{NewBound}, we check that
\begin{align*}
&\IInf |\tilde{g}'_{s}(y)|dy\leq \E\left[\int_{\int_{0}^{s}e^{\xi_{v}}dv}^{\infty}|f'(u)|du\right]\leq\int_{0}^{\infty}|f'(u)|du=\int_{\R}|f'_e(u)|du<C(f), 
\end{align*}
and 
\begin{eqnarray*}
\IInf |\tilde{g}''_{s}(y)|dy&\leq& E\left[\int_{\int_{0}^{s}e^{\xi_{v}}dv}^{\infty}\left(u-\int_{0}^{s}e^{\xi_{v}}dv\right)|f''(u)|du\right]\leq
\int_{0}^{\infty}u|f''(u)|du \\&\leq& 2\int_{\R_+}|f'_e(\ln x)|+|f''_e(\ln x)|\frac{dx}{x}=2\int_{\R}(|f'_e(y)|+|f''_e(y)|)dy<C(f), 
\end{eqnarray*}
where $C(f)$ is chosen to be the largest constant in all the inequalities above and we have used the trivial inequality $u^{2}|f''(u)|\leq |f'_e(\ln u)|+|f''_e(\ln u)|$.
Thus using all the information above we conclude that $g_s=P_s f\in \mathcal{K}$ and \eqref{Invariant set} holds. Next we consider \eqref{Bound} keeping in mind that all estimates on $\tilde{g}_s$ we used to show that $g_{s}\in\mathcal{K}$ are uniform in $s$ and $x$. We use \eqref{InfGen1} with $g(x)=xg'_{s}(x)=\tilde{g}'_s(\ln x)$, the bounds on $\tilde{g}_{s}$ and its derivatives to get
\begin{align*}
&\Big|\frac{g(x)}{x}+\E[\xi_{1}]g(x)+\frac{\sigma^{2}}{2}xg'(x)\Big|\leq C(f)x^{-1}+C(f)|\E[\xi_{1}]|+C(f)\frac{\sigma^{2}}{2}
\leq C(f,\sigma,\E[\xi_{1}])(1\wedge x^{-1}).
\end{align*}
Moreover, as in the proof of Proposition \ref{Prop1}, we can estimate
\begin{align*}
&\Big|\int_{0}^{x}g'(v)\PPPn\big(\ln{\frac{x}{v}}\big)dv\Big|+\Big|\int_{x}^{\infty}g'(v)\PPPp\big(\ln{\frac{v}{x}}\big)dv\Big|\leq\\
& \lb\PPPn(1)+\PPPp(1)\rb\int_{0}^{\infty}|g'(s)|ds+C(f)\lb \int_{0}^{1}\PPPn(s)ds+\int_{0}^{1}\PPPp(s)ds\rb= \\
&\lb\PPPn(1)+\PPPp(1)\rb\int_{-\infty}^{\infty}|\tilde{g}''(y)|dy+C(f)\lb \int_{0}^{1}\PPPn(s)ds+\int_{0}^{1}\PPPp(s)ds\rb <C
\end{align*}
and therefore \eqref{Bound} holds since
\[L^{U^{\xi}}g_{s}(x)=\frac{g(x)}{x}+\E[\xi_{1}]g(x)+\frac{\sigma^{2}}{2}xg'(x)+\int_{0}^{x}g'(v)\PPPn\big(\ln{\frac{x}{v}}\big)dv+\int_{x}^{\infty}g'(v)\PPPp\big(\ln{\frac{v}{x}}\big)dv.\]
This concludes the proof.
\end{proof}

\begin{theorem}\label{Lemma1}
Let $(\xi^{(n)})_{n\geq1}$ be a sequence of \LLPs with negative means such that
\[ \lim_{n \to \infty }\xi^{(n)}\stackrel{d}= \xi,\]
where $\xi$ is a \LLP with  $\E[\xi_{1}]<0$. Moreover, if for each $n\geq 1$, $m^{(n)}$ stands for the law of the stationary measure of the GOU process $U^{^{\xi^{(n)}}}$ defined, for any $t\geq 0,\,x\geq0$, by
\[ U^{^{\xi^{(n)}}}_{t}=xe^{\xi^{(n)}_{t}}+e^{\xi^{(n)}_{t}}\int_{0}^{t}e^{-\xi^{(n)}_{s}}ds\]
and the sequence $(m^{(n)})_{n\geq 1}$ is tight then $(m^{(n)})_{n\geq 1}$ converges weakly to $m^{(0)}$, which is the unique stationary measure of the process $U^{\xi}$, i.e.
\begin{equation}\label{L1-1}
\lim_{n\to\infty}m^{{(n)}}\stackrel{w}=m^{(0)}.
\end{equation}
\end{theorem}
\begin{proof}
Without loss of generality we assume using Skorohod-Dudley theorem, see Theorem 3.30 in Chapter 3 in \cite{Kallenberg}, that the convergence $\xi^{(n)}\rightarrow \xi$ holds a.s. in the Skorohod space $\mathcal{D}((0,\infty))$. Due to the stationarity properties of $m^{(n)}$, for each $t>0$, we have, for any $f\in C_b(\overline{\R}_+)$,
\[\left(f,m^{(n)}\right)=\left(P^{\left(n\right)}_{t}f,m^{\left(n\right)}\right)=\left(P^{\left(n\right)}_{t}f-P_{t}f,m^{\left(n\right)}\right)+\left(P_{t}f,m^{\left(n\right)}\right),\]
where $P^{\left(n\right)}_{t}$ and $P_{t}$ are the semigroups of $U^{\xi^{\left(n\right)}}_{t}$ and $U^{\xi}_{t}$.
For any $x>0$,
\begin{eqnarray}\label{ToProve2}
\left|\left(P^{\left(n\right)}_{t}f-P_{t}f,m^{\left(n\right)}\right)\right|&\leq & 2||f||_{\infty}m^{\left(n\right)}\left(x,\infty\right)+\sup_{y \leq x}\big|P^{\left(n\right)}_{t}f\left(y\right)-P_{t}f\left(y\right)\big| \nonumber \\
&\leq&
 2\big|\big|f\big|\big|_{\infty}m^{\left(n\right)}\left(x,\infty\right)+\E\left[\sup_{y \leq x}\left|f\left(U^{\xi^{(n)}}_{t}\left(y\right)\right)-f\left(U^{\xi}_{t}\left(y\right)\right)\right|\right].
\end{eqnarray}
Taking into account that $(m^{(n)})_{n\geq1}$ is tight we may fix $\delta>0$ and find $x>0$  big enough such that 
\[\sup_{n\geq 1}m^{(n)}(x,\infty)<\delta.\]
Also since  $f\in C_b(\overline{\R}_+)$ then $f$ is uniformly continuous on $\mathbb{R_{+}}$. Therefore, to show that
\[\lim_{n\to\infty}\E\left[\sup_{y \leq x}\left|f\left(U^{\xi^{(n)}}_{t}(y)\right)-f\left(U^{\xi}_{t}(y)\right)\right|\right]=0,\]
due to the dominated convergence theorem all we need to show is that
\begin{equation}\label{ToProve1}\lim_{n\to\infty}\sup_{y\leq x}|U^{\xi^{\left(n\right)}}_{t}(y)-U^{\xi}_{t}(y)|=0.\end{equation}
From the definition of $U^{\xi^{\left(n\right)}}$ and $U^{\xi}$, we   obtain that, for $y\leq x$,
\begin{align*}
&\left|U^{\xi^{\left(n\right)}}_{t}(y)-U^{\xi}_{t}(y)\right|\leq x\left|e^{\xi^{(n)}_{t}}
-e^{\xi_{t}}\right|+\left|e^{\xi^{(n)}_{t}}-e^{\xi_{t}}\right|\int_{0}^{t}e^{-\xi^{(n)}_{s}}ds+e^{\xi_{t}}\left|\int_{0}^{t}e^{-\xi^{(n)}_{s}}-e^{-\xi_{s}}ds\right|.
\end{align*}
Since $\xi^{(n)}\stackrel{a.s.}\rightarrow \xi$ in the Skorohod topology and
\[\P\left(\left\{\exists n\geq1:\: \xi^{(n)}_t-\xi^{(n)}_{t-}>0\right\} \cap \left\{ \xi_{t}-\xi_{t-}>0\right\}\right)=0\]
the first term on the right-hand side of the last expression converges a.s. to zero as $n\rightarrow\infty$. The a.s. convergence in the Skorohod space implies the existence of changes of times $(\lambda_{n})_{n\geq 1}$ such that, for each $n\geq 1$, $\lambda_{n}(0)=0$, $\lambda_{n}(t)=t$, the mapping $s \mapsto \lambda_{n}(s)$ is increasing and continuous on $[0,t]$, and
\begin{equation}\label{Skorohod3}\lim_{n\to\infty}\sup_{s\leq t}|\lambda_{n}(s)-s|=\lim_{n\to\infty}\sup_{s\leq t}\left|\lambda^{-1}_{n}(s)-s\right|=0
\end{equation}
\begin{equation}\label{Skorohod4}\lim_{n\to\infty}\sup_{s\leq t}\left|\xi^{(n)}_{\lambda_{n}(s)}-\xi_{s}\right|=\lim_{n\to\infty}\sup_{s\leq t}\left|\xi^{(n)}_{s}-\xi_{\lambda^{-1}_{n}(s)}\right|=0.
\end{equation}
Hence,
\[\Big|\int_{0}^{t}e^{-\xi^{(n)}_{s}}-e^{-\xi_{s}}ds\Big|\leq \Big|\int_{0}^{t}e^{-\xi^{(n)}_{s}}-e^{-\xi_{\lambda^{-1}_{n}(s)}}ds\Big|+\Big|\int_{0}^{t}e^{-\xi_{\lambda^{-1}_{n}(s)}}-e^{-\xi_{s}}ds\Big|.\]
The first term on the right-hand side clearly goes to zero due to \eqref{Skorohod4} whereas \eqref{Skorohod3} implies that the second term goes to zero a.s. due to the dominated convergence theorem and the fact that pathwise, for $s\leq t$,
\[\limsup_{n\to\infty}\left|e^{-\xi_{\lambda^{-1}_{n}(s)}}-e^{-\xi_{s}}\right|>0\]
only on the set of jumps of $\xi$ and this set has a zero Lebesgue measure.
Thus we conclude that
\[\lim_{n\to\infty}e^{\xi_{t}}\left|\int_{0}^{t}e^{-\xi^{(n)}_{s}}-e^{-\xi_{s}}ds\right|=0.\]

Similarly we observe that
\begin{equation}\label{3}
\lim_{n\to\infty}\left|e^{\xi^{(n)}_{t}}-e^{\xi_{t}}\right|\int_{0}^{t}e^{-\xi^{(n)}_{s}}ds\leq \lim_{n\to\infty}t\left|e^{\xi^{(n)}_{t}}-e^{\xi_{t}}\right|e^{\sup_{s\leq t}(-\xi^{(n)}_{s})}=0,
\end{equation}
where the last identity follows from
\[\sup_{s\leq t}\left|(-\xi^{(n)}_{s})\right|\leq \sup_{s\leq t}\left|\xi_{\lambda^{-1}_{n}(s)}\right|+\sup_{s\leq t}\left|\xi^{(n)}_{s}-\xi_{\lambda^{-1}_{n}(s)}\right|=\sup_{s\leq t}\left|\xi_{s}\right|+\sup_{s\leq t}\left|\xi^{(n)}_{s}-\xi_{\lambda^{-1}_{n}(s)}\right|\]
and an application of \eqref{Skorohod4}.
Therefore, \eqref{ToProve1} holds and
\[\lim_{n\to\infty}\sup_{y \leq x}\left|f\left(U^{(n)}_{t}(y)\right)-f\left(U^{\xi}_{t}(y)\right)\right|=0.\]
The dominated convergence theorem then easily gives that the right-hand side of \eqref{ToProve2} goes to zero and hence
\[\limsup_{n\to\infty}\left|\left(P^{(n)}_{t}f-P_{t}f,m^{(n)}\right)\right|\leq 2||f||_{\infty}\sup_{n\geq 1}m^{(n)}(x,\infty)\leq 2||f||_{\infty}\delta.\]
As $\delta>0$ is arbitrary we show that
\[\lim_{n\to\infty}\left|\left(P^{(n)}_{t}f-P_{t}f,m^{(n)}\right)\right|=0.\]
Since $(m^{(n)})_{n\geq 1}$ is tight we choose a subsequence $(m^{(n_{k})})_{k\geq 1}$ such that $\lim_{k\to\infty}m^{(n_{k})}\stackrel{d}=\nu$ with $\nu$ a probability measure. Then, for each $t\geq 0$,
\[(f,\nu)=\lim_{k\to\infty}\left(f,m^{(n_{k})}\right)=\lim_{k\to\infty}\left(P^{(n_{k})}_{t}f,m^{(n_{k})}\right)=\lim_{k\to\infty}\left(P_{t}f,m^{(n_{k})}\right)=\left(P_{t}f,\nu\right).\]
Therefore $\nu$ is a stationary measure for $U^{\xi}$. But since  $m^{(0)}$ is the unique stationary measure we conclude that
\[\lim_{n\to\infty}m^{(n)}{\stackrel{w}{=}} \nu= m^{(0)}.\]
This translates to the proof of \eqref{L1-1}.
\end{proof}

\section{Proof of Theorem \ref{MainTheorem}} \label{proof:mt}
We start the proof by collecting some useful properties in two trivial lemmas. The first one discusses the properties of $\Psi$.
\begin{lemma}{\cite[Theorem 25.17]{Sato-99}}\label{PrelimLemma1}
The function $\Psi$, defined in \eqref{psi}, is always well-defined on $i\mathbb{R}$. Moreover, $\Psi$ is analytic on the strip $\{z\in\mathbb{C};\,-a_-<\Re(z)<a_+\}$, where $a_-,a_+>0$  if and only if $\E\left[e^{(-a_-+\epsilon)\xi_{1}}\right]<\infty$ and $\E\left[e^{(a_+-\epsilon)\xi_{1}}\right]<\infty$ for all $0<\epsilon<a_-\wedge a_+$.
\end{lemma}

The second lemma concerns the properties of $\phi_{\pm}$ and is easily obtained using Lemma \ref{PrelimLemma1}, \eqref{eq:wh} together with the analytical extension and the fact that subordinators have all negative exponential moments.
\begin{lemma}\label{PrelimLemma2}
Let $\xi$ be a \LLP with $\E[\xi_{1}]<\infty$. Then $\phi_{+}$ is always analytic on the strip $\{z\in\mathbb{C};\: \Re(z)<0\}$ and is well-defined on $i\mathbb{R}$. Moreover $\phi_{+}$ is analytic on $\{z\in\mathbb{C}; \:\Re(z)<a_+\}$, for $a_+\geq 0$, if and only if  $\E\left[e^{(a_+-\epsilon)\xi_{1}}\right]<\infty$, for some $\epsilon>0$. Similarly $\phi_{-}$ is always analytic on the strip $\{z\in\mathbb{C}; \:\Re(z)>0\}$ and is well-defined on $i\mathbb{R}$ and $\phi_{-}$ is analytic on $\{z\in\mathbb{C};\: \Re(z)< -a_-\}$, for $a_-\geq 0$, if and only if $\E[e^{(-a_-+\epsilon)\xi_{1}}]<\infty$, for some $\epsilon>0$.
Finally, the  Wiener-Hopf factorization \eqref{eq:wh} holds
on the intersection of the strips where $\phi_{+}$ and $\phi_{-}$ are well-defined.
\end{lemma}

\subsection{ Proof in the case ${\bf{E}_+}$}
We recall that in this part we assume, in particular, that $\xi$ is a L\'evy process with a finite negative mean and that there exists $a_+>0$ such that $|\Psi(z)|<\infty$ for any $0<\Re(z)<a_+$.
 Next, we write $\theta^*= \max(\theta,a_+)$, where $\theta=\inf \{s>0; \: \Psi(s)=0\}$ (with the convention that $\inf \emptyset =+\infty$). We also recall from \cite{Carmona-Petit-Yor-97}, see also \cite{Maulik-Zwart-06}, that the Mellin transform of ${\rm{I}}_{\xi}$ defined by
 \[\mathcal{M}_{m_{\xi}}(z)=\int_{0}^{\infty}x^{z-1}m_{\xi}(x)dx\]
satisfies, for any $0<\Re(z)<\theta^*$, the following functional equation
\begin{equation}\label{Maulik}
\mathcal{M}_{m_{\xi}}(z+1)=-\frac{z}{\Psi(z)}\mathcal{M}_{m_{\xi}}(z).
\end{equation}
We proceed by proving  the following easy result.
\begin{lemma}\label{lem:ey}
If $\mu_+ \in \mathcal{P}$ then there exists a spectrally positive L\'evy process $Y$ with Laplace exponent $\psi_+(-s)=-s\phi_{+}(-s),\: s\geq0,$ and a negative finite mean $-\phi_+(0)$. Moreover, if $\xi$ has a negative finite mean   then $\E\left[\left({\rm{I}}_{H^-}\times{\rm{I}}_{Y}\right)^{-1}\right]=-\phi_+(0)\phi_-'(0^+)<+\infty$.
\end{lemma}
\begin{proof}
The first claim follows readily from \cite[Theorem VII.4(ii)]{Bertoin-96} and by observing that  $\psi_+'(0^-)=\phi_+(0)$. From \eqref{eq:msn} we get that $\E\left[{\rm{I}}_{Y}^{-1}\right] =k_+$. Next, since $-\infty<\E[\xi_1]<0$, using the dual version of \cite[Corollary 4.4.4(iv)]{Doney}, we get that $-\infty<\E[H^-_1]<0$ and thus $¦\phi'_-(0^+)¦<\infty$.  From the functional equation \eqref{Maulik}, we easily deduce that $\E\left[{\rm{I}}_{H^-}^{-1}\right]=-\phi^{\prime}_-(0^+)$ which completes the proof since the two random variables are independent.
\end{proof}

\begin{lemma}\label{AuxLemma1}
Assume that $\xi$ has a finite negative mean and condition ${\bf{E}}_+$ holds. Let $\eta$ be a positive random variable with density $\kappa(x)$, such that $\E[\eta^{-1}]<\infty$ and $\E[\eta^{\delta}]<\infty$, for some $\theta^{*}>\delta>0$. Then, for any $z$ such that $\Re(z)\in (0,\delta)$,
\begin{equation}\label{New111}
\mathcal{M}_{\mathcal{L}\kappa}(z)=\int_{0}^{\infty}x^{z-1}\mathcal{L}\kappa(x)dx=\frac{\Psi(z)}{z^{2}}\mathcal{M}_{\kappa}(z+1)+\frac{1}{z}\mathcal{M}_{\kappa}(z)
\end{equation}\label{g(x)}
and if $\mathcal{M}_{\mathcal{L}\kappa}(z)=0$, for $0<a<\Re(z)<b<\delta$, then $\mathcal{L}\kappa(x)=0$ a.e..

 Furthermore the law of the positive random variable  ${\rm{I}}_{Y}\times{\rm{I}}_{H^{-}}$, as defined in Theorem \ref{MainTheorem},  is absolutely continuous with a density, denoted by $\overline{m}$, which satisfies
\begin{equation}\label{k-tilde1}
\mathcal{L}\overline{m}(x)= 0 \text{ for a.e. $x>0$}.
\end{equation}
\end{lemma}
\begin{remark}
Note that the proof of this lemma shows that we have uniqueness for the probability measures with first negative moment that satisfy \eqref{Maulik}. This is a rather indirect approach and seems to be more general than the verification approach of \cite{Kuznetsov-Pardo-11}, see Proposition 2, where precise knowledge on the rate of decay of the Mellin transform $\mathcal{M}_{m_{\xi}}(z)$ is needed. In general such an estimate on the decay seems impossible to obtain.
\end{remark}
\begin{proof} We start by proving \eqref{New111}. Note that since $\int_0^{\infty}y^{-1}\kappa(y)dy<\infty$, we can use  Proposition \ref{Prop1} to get
\begin{equation}\label{g(x)}
\mathcal{L}\kappa(x)=\frac{\sigma^{2}}{2}x\kappa(x)+\int_{x}^{\infty}\frac{\kappa(y)}{y}dy+\E[\xi_{1}]\int_{x}^{\infty}\kappa(y)dy
+\int_{x}^{\infty}\PPPn\left(\ln{\frac{y}{x}}\right)\kappa(y)dy+
\int_{0}^{x}\PPPp\left(\ln{\frac{x}{y}}\right)\kappa(y)dy.
\end{equation}
As $\kappa$ is a density, one can use Fubini Theorem to get, after some easy computations, that
for any  $\epsilon<\Re(z)<\delta<\theta^{*}$, with $0<\epsilon<\delta$,
\begin{eqnarray*}
\mathcal{M}_{\mathcal{L}\kappa}(z)&=&\int_{0}^{\infty}x^{z-1}\mathcal{L}\kappa(x)dx\\
&=& \mathcal{M}_{\kappa}(z+1)\left(\frac{\sigma^{2}}{2}+\frac{\E[\xi_{1}]}{z}+\int_{0}^{\infty}\PPPn(y)e^{-zy}dy+\int_{0}^{\infty}\PPPp(y)e^{zy}dy \right)+\frac{1}{z}\mathcal{M}_{\kappa}(z)\\
&=&\frac{\Psi(z)}{z^{2}}\mathcal{M}_{\kappa}(z+1)+\frac{1}{z}\mathcal{M}_{\kappa}(z).
\end{eqnarray*}
Let $\mathcal{M}_{\mathcal{L}\kappa}(z)=0$ for $\epsilon<\Re(z)<\delta$. We show using that all terms in \eqref{g(x)} are positive except the negative one due to $\E[\xi_{1}]<0$ that, with $u=\Re(z)$,
\begin{eqnarray*}
\int_{0}^{\infty}x^{u-1}\left|\mathcal{L}\kappa(x)\right|dx &\leq&
\mathcal{M}_{\kappa}(u+1)\left(\frac{\Psi(u)}{u^{2}}-2\frac{\E[\xi_{1}]}{u}\right)+\frac{1}{u}\mathcal{M}_{\kappa}(u)<\infty.
\end{eqnarray*}
Given the absolute integrability of $x^{z-1}\mathcal{L}\kappa(x)$ along imaginary lines determined by $\epsilon<\Re(z)<\delta$ we can apply the Mellin inversion theorem to the identity  $\mathcal{M}_{\mathcal{L}\kappa}(z)=0$ to get $\mathcal{L}\kappa(x)=0$ a.e., see Theorem $6$ in Section $6$ in \cite{Butzer-Jansche-97}.

Next it is plain that the law of ${\rm{I}}_{Y}\times{\rm{I}}_{H^{-}}$ is absolutely continuous since, for any $x>0$,
\begin{equation}\label{k-tilde}
\overline{m}(x)=\int_{0}^{\infty}m_{Y}\Big(\frac{x}{y}\Big)y^{-1}m_{H^{-}}(y)dy.
\end{equation}
Furthermore from the Wiener-Hopf factorization  \eqref{eq:wh} and the definition of $\psi_+$, we have that
\[\frac{-z}{\Psi(z)}=\frac{-z}{\phi_-(z)}\frac{-z}{\psi_+(z)}\]
which is valid, for any  $0<\Re(z)<\theta^{*}$.
Thus, we deduce from the functional equation \eqref{Maulik} and the independency of  $Y$ and $H^-$ that, for any  $0< \Re(z)<\theta^{*}$,
\begin{equation}\label{Maulik2}
\mathcal{M}_{\overline m}(z+1)=-\frac{z}{\Psi(z)}\mathcal{M}_{\overline m}(z).
\end{equation}
Next, since from Lemma \ref{lem:ey}, we have that $\int_0^{\infty}y^{-1}\overline{m}(y)dy<\infty$, we can use  Proposition \ref{Prop1} and thus \eqref{g(x)} and subsequently \eqref{New111} are valid for $\overline{m}$.
Moreover due to the representation \eqref{psi} of $\Psi$ and relation \eqref{Maulik2} we have that
for any  $\epsilon<\Re(z)<\theta^{*}$ with $0<\epsilon<\theta^{*}/4$,
\begin{eqnarray*}
 \mathcal{M}_{\mathcal{L}\overline{m}}(z)&=&\frac{\Psi(z)}{z^{2}}\mathcal{M}_{\overline{m}}(z+1)+\frac{1}{z}\mathcal{M}_{\overline{m}}(z)=0
\end{eqnarray*}
and we conclude that $\mathcal{L}\overline{m}(x)=0$ a.e.

\end{proof}

We are now ready to complete the proof of Theorem \ref{MainTheorem} in the case ${\bf{E}}_+$. Indeed, since $m_{\xi}$, the density of ${\rm{I}}_{\xi}$, is the density of the stationary measure of $U^{\xi}$, we have that $m_{\xi}$ is also solution to \eqref{IntegralEqn}. Combining Lemma \ref{AuxLemma1} with the uniqueness argument of Theorem \ref{O-UTheorem}, we conclude that the factorization \eqref{MainAssertion} holds.

\subsection{Proof of the two other cases : ${\bf P}+$ and ${\bf P_\pm}$}
We start by providing some results which will be used several times throughout  this part.
\begin{proposition}[Carmona et al.~\cite{Carmona-Petit-Yor-97}] \label{prop:ms}
Let $H$ be the negative of a (possibly killed) subordinator with Laplace exponent $\phi$, then the law of ${\rm{I}}_{H}$ is determined by its positive entire moments as follows
\begin{eqnarray}\label{eq:ms}
\E[{\rm{I}}_{H}^m] &=&\frac{\Gamma(m+1)}{\prod_{k=1}^{m}\left(-\phi(k)\right)
}, \: m=1,2,\ldots
\end{eqnarray}
\end{proposition}

\begin{proposition}[Bertoin and Yor \cite{Bertoin-Yor-02}]\label{prop:msp}
Let $Y$ be an unkilled  spectrally positive L\'evy process with a negative mean and Laplace exponent $\psi_+$, then the law of $1/{\rm{I}}_{Y}$ is determined by its positive entire moments as follows
\begin{eqnarray} \label{eq:msn}
\E[{\rm{I}}_{Y}^{-m}] &=&\E[-Y_1]\frac{\prod_{k=1}^{m-1}\psi_+(-k)}{\Gamma(m)}, \: m=1,2,\ldots,
\end{eqnarray}
with the convention that the right-hand side is $\E[-Y_1]$ when $m=1$.
\end{proposition}
In order to get \eqref{MainAssertion} in the case when $\xi$ does not have some finite positive exponential moments, we will develop  some approximation techniques. However, the exponential functional is not continuous in the Skorohod topology and therefore we have to find some criteria in order to secure  the weak convergence  of sequences of exponential functionals. This is the aim of the next result.

\begin{lemma}\label{Lemma111}

Let $(\xi^{(n)})_{n\geq1}$ be a sequence of \LLPs with negative means such that
\[ \lim_{n\to\infty}\xi^{(n)}\stackrel{d}= \xi\]
where $\xi$ is a \LLP with  $\E[\xi_{1}]<0$.  Let us assume further that at least one of the following conditions holds:
\begin{enumerate}[(a)]
\item for each $n\geq 1$, $\xi^{(n)}$ and $\xi$ are unkilled spectrally positive L\'evy processes such that $\lim_{n\to \infty}\E[\xi_1^{(n)}] =\E[\xi_1],$ 
\item for each $n\geq 1$, $\xi^{(n)}$ and $\xi$ are the negative of unkilled subordinators,
\item  the sequence $(m_{\xi^{(n)}})_{n\geq 1}$ is tight, where $m_{\xi^{(n)}}$ is the law of ${\rm{I}}_{\xi^{(n)}}.$
\end{enumerate}
Then, in all cases, we have
\begin{equation}\label{L1-111}
\lim_{n\to\infty}{\rm{I}}_{\xi^{(n)}}\stackrel{d}={\rm{I}}_{\xi}.
\end{equation}
\end{lemma}
\begin{proof}
 To prove \eqref{L1-111} in the case (a), we simply observe that writing $\psi^{(n)}_+$ for the Laplace exponent of $\xi_1^{(n)}$, we have,  by L\'evy continuity Theorem, see e.g. \cite[Theorem XIII.1.2]{Feller-71}, that  for all $s\geq0$, $\psi^{(n)}_+(-s)\rightarrow \psi_+(-s)$ as  $n \rightarrow \infty$.    Next, putting $M^{(n)}_m$ for the sequence of negative entire moments of ${\rm{I}}_{\xi^{(n)}}$, we easily deduce, from \eqref{eq:msn} for all $m=1,2\ldots,$ that  $\lim_{n\to \infty}M^{(n)}_m = M_m$ where $M_m$ is the sequence of negative entire moments of ${\rm{I}}_{\xi}$. These random variables being moment determinate, see Proposition \ref{prop:msp}, we conclude (a) by invoking  \cite[Examples (b) p.269]{Feller-71}. The second case follows by applying a similar line of reasoning to the expression \eqref{eq:ms}. Finally,  the case (c) is a straightforward  consequence of \eqref {L1-1} of Theorem \ref{Lemma1}.
\end{proof}
Before stating our next result, we need to introduce the following notation.  Let us first recall that the reflected processes
 $\left(R^+_t=\sup_{0\leq s\leq t}\xi_s-\xi_t\right)_{t\geq 0}$ and $\left(R^-_t=\xi_t-\inf_{0\leq s\leq t}\xi_s\right)_{t\geq 0}$ are Feller processes in $[0,\infty)$  which possess local times $L^{\pm}=(L^{\pm}_t)_{t\geq0}$ at the level $0$.  The ascending and descending ladder times, $l^{\pm}=(l^{\pm}(t))_{t\geq0}$, are defined as the right-continuous inverses of $L^{\pm}$, i.e. for any $t\geq0$, $l^{\pm}(t)=\inf\{s> 0;\:  L^{\pm}_s>t\}$
 and the ladder height processes
 $H^+=(H^+(t))_{t\geq0}$ and $-H^-=(-H^-(t))_{t\geq0}$ by
 $$H^+(t)=\xi_{l^{+}(t)}=\sup_{0\leq s\leq l^{+}(t)}\xi_s\,, \qquad \hbox{ whenever } l^{+}(t)<\infty\,,$$
  $$-H^-(t)=\xi_{l^{-}(t)}=\inf_{0\leq s\leq l^{-}(t)}\xi_s\,, \qquad \hbox{ whenever } l^{-}(t)<\infty\,.$$
Here, we use the convention $\inf\{\varnothing\} =\infty$ and $H^{+}(t)=\infty$ when $L^{+}_{\infty}\leq t$ and $-H^{-}(t)=-\infty$ when $L^{-}_{\infty}\leq t$.
From \cite[p. 27]{Doney}, we have, for $\alpha,\beta\geq 0$,
 \begin{equation}\label{BivLadder}
\log{ \E\left[e^{-\alpha l^{+}(1)-\beta H^{+}(1)}\right]} = -k(\alpha,\beta)=-k_{+}-\eta_{+}\alpha-\delta_{+}\beta-\int_{0}^{\infty}\int_{0}^{\infty}\Big(1-e^{-(\alpha y_{1}+\beta y_{2})}\Big)\mu_{+}(dy_{1},dy_{2}),
 \end{equation}
 where $\eta_{+}$ is the drift of the subordinator $l^{+}$ and $\mu_{+}(dy_{1},dy_{2})$ is the \LL measure of the bivariate subordinator $(l^{+},H^{+})$. Similarly, for $\alpha,\beta\geq 0$,
 \begin{equation}\label{BivLadder1}
\log  \E\left[e^{-\left(\alpha l^{-}(1)-\beta H^{-}(1)\right)}\right]=-k_{*}(\alpha,\beta)=-\eta_{-}\alpha-\delta_{-}\beta-\int_{0}^{\infty}\int_{0}^{\infty}\Big(1-e^{-(\alpha y_{1}+\beta y_{2})}\Big)\mu_{-}(dy_{1},dy_{2}),
 \end{equation}
 where $\eta_{-}$ is the drift of the subordinator $l^{-}$ and $\mu_{-}(dy_1,dy_2)$ is the \LL measure of the bivariate subordinator $(l^{-},-H^{-})$.

\begin{lemma}\label{Lemma3}
Let $\xi$ be a \LLP with triplet $(a,\sigma,\Pi)$ and Laplace exponent $\psi$. Let,  for any $n\geq 1$,  $\xi^{(n)}$ be the  \LLP  with Laplace exponent denoted by  $\psi^{(n)}$ and triplet $(a,\sigma,\Pi^{(n)})$ such that $\Pi^{(n)}=\Pi$ on $\R_{-}$ and on $\R_{+}$
\[ \Pi^{(n)}(dy)= h^{(n)}(y)\Pi(dy),\]
 where for all $y>0$, $0\leq h^{(n)}(y) \uparrow 1$ as $n\rightarrow \infty$ and uniformly for $n\geq 1$ we have that for some $C\geq0$,  $\limsup_{y\to 0}y^{-1}(1-h_{n}(y))\leq C$. Then,
 \begin{equation}\label{eqn:Referee1}
 \lim_{n\to\infty} \xi^{(n)}\stackrel{d}=\xi,
 \end{equation}
and for all $\alpha\geq 0,\,\beta\geq0$, we have, as $n\rightarrow \infty $,
\begin{align}\label{LadderHeight}
k^{(n)}(\alpha,\beta)  \rightarrow k (\alpha,\beta) ,\\
k_*^{(n)}(\alpha,\beta)  \rightarrow k_* (\alpha,\beta), \nonumber
\end{align}
where  $k^{(n)}(\alpha,\beta)$ and $k_{*}^{(n)}(\alpha,\beta)$ stand for the bivariate Laplace exponents of the ladder processes of $\xi^{(n)}$, normalized such that  $k^{(n)}(1,0)=k_{*}^{(n)}(1,0)=1$. Also $k (\alpha,\beta)$ and $k_* (\alpha,\beta)$ stand for the  bivariate Laplace exponents of the ladder processes of $\xi$, normalized such that  $k(1,0)=k_{*}(1,0)=1$.
\end{lemma}
\begin{remark}
Denote by $\left( l^{+}_{(n)},H^+_{(n)}\right)$ $\left(\text{resp.~}\left( l^{-}_{(n)},-H^{-}_{(n)}\right)\right)$ the bivariate ascending (resp.~descending) ladder processes  of  $\xi^{(n)}$ and $\left( l^{+},H^+\right)$ $\left(\text{resp.~}\left( l^{-},-H^{-}\right)\right)$ the bivariate ascending (resp.~descending) ladder processes  of  $\xi$, then from the L\'evy continuity Theorem we deduce that as $n\rightarrow \infty $,
\begin{align}\label{LadderHeight}
\nonumber &\left( l^{+}_{(n)},H^+_{(n)}\right)\stackrel{d}{\rightarrow }\left( l^{+},H^{+}\right),\\
 &\left( l^{-}_{(n)},-H^-_{(n)}\right)\stackrel{d}{\rightarrow }\left( l^{-},-H^{-}\right),
\end{align}
where in the convergence the sequence of killing rates of the ladder height processes also converge to the killing rate of the limiting process.
\end{remark}

\begin{proof}

For the sake of completeness and also to include both the compound Poisson case and  the cases when $\alpha=0$ and/or $\beta=0$, we must improve  the proof of Lemma 3.4.2 in \cite{Vigon}. Next, since $\Pi^{(n)}(dx) \stackrel{v}{\rightarrow} \Pi(dx)$,  where $\stackrel{v}{\rightarrow}$ stands for the vague convergence, we get \eqref{eqn:Referee1} from e.g. \cite[Theorem 13.14(i)]{Kallenberg}.  We note the identity
\begin{equation}\label{eq:id-t}
\xi\stackrel{d}{=} \xi^{(n)}+\tilde{\xi}^{(n)},
\end{equation} where $\tilde{\xi}^{(n)}$ is a subordinator with L\'evy measure $\tilde{\Pi}^{(n)}(dy)=(1-h_{n}(y)) \Pi(dy)$ and no drift, since $1-h_{n}(y)=O(y)$ at zero.  Then, when $\xi$ is  a compound Poisson process  we have that $\tilde{\xi}^{(n)}$ is a compound Poisson process and, for all $t>0$,
\[ \P\left(\xi^{(n)}_t=0\right)=\P\left(\xi_t=0, t<\tilde{T}^{(n)}\right)+\P\left(\xi^{(n)}_t=0, t\geq \tilde{T}^{(n)}\right),\]
where $\tilde{T}^{(n)}=\inf\{ s>0; \: \tilde{\xi}^{(n)}_s >0\}$.  Since for all $y>0$, $h^{(n)}(y) \uparrow 1$, then $\P( t> \tilde{T}^{(n)}) \rightarrow 0$ as $n\rightarrow \infty$ and
\[ \P\left(\xi^{(n)}_t \in dy\right)\mathbb{I}_{\{y\geq0\}} \stackrel{v}{\rightarrow} \P\left(\xi_t \in dy\right)\mathbb{I}_{\{y\geq0\}}.\]
 When $\xi$ is not a compound Poisson process,  the law of $\xi^{(n)}$ does not charge $\{0\}$ and thus as $n\rightarrow \infty$
\[ \P\left(\xi^{(n)}_t \in dy\right)\mathbb{I}_{\{y>0\}} \stackrel{v}{\rightarrow} \P\left(\xi_t \in dy\right)\mathbb{I}_{\{y>0\}}.\]
 Henceforth, from  the expression
\begin{equation}\label{eq:def-biv} k^{(n)}(\alpha,\beta) = \exp\left(\int_0^{\infty}dt\int_0^{\infty}\left(e^{-t} -e^{-\alpha t -\beta y} \right)t^{-1}\P(\xi^{(n)}_t \in dy)\right)
\end{equation}
which holds for any $\alpha>0$ and $\beta>0$, see e.g.~\cite[Corollary VI.2.10]{Bertoin-96},  we deduce easily that for both cases
\begin{equation}  \label{eq:cv-be}
\lim_{n\to \infty} k^{(n)}(\alpha,\beta) = k(\alpha,\beta).
\end{equation} Moreover, we can write
\begin{equation}\label{eq:def-biva}
 k^{(n)}(\alpha,\beta) =k^{(n)}(0,0)+\tilde{k}^{(n)}(\alpha,\beta),
 \end{equation}
 where $\tilde{k}^{(n)}$ are the Laplace exponents of unkilled  bivariate  subordinators,  see \cite[p.~27]{Doney}.  Note from \eqref{eq:def-biv} that
\begin{equation*}k^{(n)}(0,0) =  \exp\left(-\int_0^{\infty}\left(1-e^{-t}\right)\P\left(\xi^{(n)}_t \geq 0\right)\frac{dt}{t}\right).
\end{equation*}
Next from \eqref{eq:id-t} and the fact that $\tilde{\xi}^{(n)}$ is a subordinator, we have that $\P\left(\xi^{(n)}_t \geq 0\right) \leq \P\left(\xi_t \geq 0\right) $ and appealing to the monotone convergence theorem we get that $k^{(n)}(0,0) \downarrow k(0,0)$. Hence we deduce from \eqref{eq:cv-be} and \eqref{eq:def-biva} that for any $\alpha,\beta >0$, $\tilde{k}^{(n)}(\alpha,\beta) \rightarrow \tilde{k}(\alpha,\beta)$ where $ \tilde{k}(\alpha,\beta)=k(\alpha,\beta) -k(0,0)$. From the L\'evy continuity theorem, we have, writing $\left(\tilde{l}_{(n)}^{+},\tilde{H}_{(n)}^{+}, \right)$ for the unkilled versions of the ascending bivariate ladder processes, that $\left(\tilde{ l}^{+}_{(n)},\tilde{H}^+_{(n)}\right)\stackrel{d}{\rightarrow }\left(\tilde{l}^+,\tilde{H}^{+}\right)$, where $\left(\tilde{l}^+,\tilde{H}^{+}\right)$ stands also for the  unkilled version of  $\left(\tilde{l}^+,\tilde{H}^{+}\right)$. These probability distributions being proper, we have that for all $\alpha,\beta \in \R$, $\tilde{k}^{(n)}(i\alpha,i\beta) \rightarrow \tilde{k}(i\alpha,i\beta)$, see \cite[Theorem XV.3.2]{Feller-71}. Hence $k^{(n)}(0,i\beta) \rightarrow k(0,i\beta)$ for all $\beta \in \R$ which completes the proof for the ascending ladder height processes. The proof of the convergence of the Laplace exponent of the bivariate descending ladder process follows readily from the identities
\begin{eqnarray*}
\psi^{(n)}(i\beta)-\alpha&=&-k^{(n)}(\alpha,-i\beta)k^{(n)}_*(\alpha,i\beta) \\
\psi(i\beta)-\alpha&=&-k(\alpha,-i\beta)k_*(\alpha,i\beta)
\end{eqnarray*}
and the convergence of $\psi^{(n)}$ to $\psi$ and $k^{(n)}$ to  $k$.
\end{proof}

\subsubsection{The case ${\bf P}+$}
We first consider the case when  $\xi$ satisfies both the conditions ${\bf P}+$  and $\E[\xi_1]>-\infty$.  We start by showing that   the condition ${\bf P}+$ implies that $\mu_+ \in \mathcal{P}$. To this end, we shall need the so-called \emph{equation amicale invers\'ee} derived by Vigon, for all $x>0$,
\begin{equation} \label{eq:ami-inv}
\bar{\mu}_+(x)=\int_{0}^{\infty}\PPp(x+y)\mathcal{U}_-(dy),
\end{equation}
where $\mathcal{U}_-$ is the renewal measure corresponding to the subordinator $H^-$, see e.g. \cite[Theorem 5.16]{Doney}.
\begin{lemma}\label{Vigon}
Let us assume that $\PPp(x)$ has a non-positive derivative $\pi_+(x)$ defined for all $x>0$ and such that $-\pi_+(x)$ is non-increasing. Then $\bar\mu_+(x)$ is differentiable with derivative $u(x)$ such that $-u(x)$ is non-increasing.
\end{lemma}
\begin{proof}
 Fix $x>0$ and choose $0<h<x/3$. Then we have the trivial bound using the non-increasing property of $-\pi_+(x)$ and the description \eqref{eq:ami-inv} of $\bar{\mu}_+(x)$
\begin{eqnarray*}
\frac{\left|\bar\mu_+(x\pm h)-\bar\mu_+(x)\right|}{h}&\leq& \int_{0}^{\infty}\frac{\left|\PPp(x+y\pm h)-\PPp(x+y)\right|}{h}\mathcal{U}_-(dy)\\ &\leq& \int_{0}^{\infty}\left(-\pi_+\left(x+y-h\right)\right)\mathcal{U}_-(dy)\\ &\leq& \int_{0}^{\infty}\left(-\pi_+\left(\frac{2x}{3}+y\right)\right)\mathcal{U}_-(dy).
\end{eqnarray*}
We show now that the last expression is finite. Note that
\[\int_{0}^{\infty}\left(-\pi_+\left(\frac{2x}{3}+y\right)\right)\mathcal{U}_-(dy)\leq \sum_{n\geq 0}-\pi_+\left(\frac{2x}{3}+n\right)\left(\mathcal{U}_-(n+1)-\mathcal{U}_-(n)\right).\]
From the trivial inequality  $\mathcal{U}_-(n+1)-\mathcal{U}_-(n)\leq \mathcal{U}_-(1)$, see \cite[ Chapter 2, p.11]{Doney}, and since $-\pi_+(x)$ is the non-increasing density of $\PPp(x)$, we have with $C=\mathcal{U}_-(1)>0$,
\begin{eqnarray*}
\int_{0}^{\infty}-\pi_+\left(\frac{2x}{3}+y\right)\mathcal{U}_-(dy) & \leq & C\sum_{n\geq 0}-\pi_+\left(\frac{2x}{3}+n\right)\\
&\leq&
-C\pi_+\left(\frac{2x}{3}\right)+C\sum_{n\geq 1}\left(\PPp\left(\frac{2x}{3}+n-1\right)-\PPp\left(\frac{2x}{3}+n\right)\right)\\
&\leq& -C\pi_+\left(\frac{2x}{3}\right)+C\PPp\left(\frac{2x}{3}\right)<\infty.
\end{eqnarray*}
Therefore, for all $x>0$, the dominated convergence applies and gives
\[u(x)=\int_{0}^{\infty}\pi_+(x+y)\mathcal{U}_-(dy).\]
As $-\pi_+(x)$ is non-increasing we deduce that $-u(x)$ is non-increasing as well.
\end{proof}

In the case ${\bf P}+$, in comparison to the case ${\bf E}_+$, we have that $\xi$ does not necessarily have some positive exponential moments.
To circumvent this difficulty we introduce the  sequence of \LLPs  $\xi^{(n)}$ obtained from $\xi$ by the following construction: we keep the negative jumps intact and we discard some of the positive ones. More precisely, we thin the positive jumps of $\xi$ to get a \LLP $\xi^{(n)}$ with $\PPp^{(n)}$  whose density has the form
\begin{align}\label{modifiedPi}
&\pi^{(n)}_+(x)=\pi_+(x)\left(\mathbb{I}_{\{0<x\leq 1\}}+e^{-n^{-1}(x-1)}\mathbb{I}_{\{x>1\}}\right).
\end{align}
Clearly, $-\pi^{(n)}_+(x)$ is non-increasing and $\E\left[e^{s\xi^{(n)}_{1}}\right]<\infty$, for $s\in(0,n^{-1})$, see \eqref{modifiedPi}. Moreover, since we have only thinned the positive jumps and pointwise $\lim_{n\to\infty}\pi^{(n)}_+(x)=\pi_+(x)$, see \eqref{modifiedPi},
\begin{equation}\label{convergence}
   \lim_{n\to\infty}\xi^{(n)}\stackrel{a.s.}=\xi
\end{equation}
 almost surely in the Skorohod space $\mathcal{D}(0,\infty)$.
Finally, since $-\infty<\E\left[\xi^{(n)}_{1}\right]<\E\left[\xi_{1}\right]<0$ and  $-\pi^{(n)}_+(x)$ is non-increasing then Lemma \ref{Vigon} applies and we deduce that the \LL measure of the ascending ladder height process of $\xi^{(n)}$ has a negative density whose absolute value is non-increasing in $x$. Then since, for each $n\geq1$, $\xi^{(n)}$ has some finite positive exponential moments, we have that
\begin{equation}\label{approx}
 {\rm{I}}_{\xi^{(n)}}\stackrel{d}={\rm{I}}_{H_{(n)}^{-}} \times {\rm{I}}_{Y^{(n)}}.
 \end{equation}
 Since we thinned the positive jumps of $\xi$, for all $t\geq 0$, $\xi^{(n)}_{t}\leq \xi_{t}$ and the monotone convergence theorem together with \eqref{convergence} imply that
\begin{equation}\label{limit}
\lim_{n\to\infty}{\rm{I}}_{\xi^{(n)}}\stackrel{a.s.}={\rm{I}}_{\xi}.
\end{equation}
By the choice of the approximating sequence $\xi^{(n)}$ we can first use Lemma \ref{Lemma3} to get
\begin{equation}\label{Ingredient1}
\lim_{n\to\infty}H_{(n)}^{-}\stackrel{d}=H^{-}
\end{equation}
and then Lemma \ref{Lemma1} (b) to obtain that
\begin{equation}\label{ConvSub}
\lim_{n\to\infty}{\rm{I}}_{H_{(n)}^{-}}\stackrel{d}={\rm{I}}_{H^{-}}.
\end{equation}
Again from  Lemma \ref{Lemma3} we deduce that $k^{(n)}(0,-s) \rightarrow k(0,-s)$, for all $s\geq0$, and $\lim_{n\to\infty}\E[Y^{(n)}_1]=-\lim_{n\to\infty}k^{(n)}(0,0)=\E[Y_1]$, so we can apply Lemma \ref{Lemma1} (a) to get that
\[\lim_{n\to\infty}{\rm{I}}_{Y^{(n)}}\stackrel{d}={\rm{I}}_{Y},\]
which completes the proof in this case. \QED

\subsubsection{The case ${\bf P}_{\pm}$}
First from the philanthropy theory developed by  Vigon \cite{Vigon}, we know that the conditions $\mu_+ \in  \mathcal{P}$ and $\mu_- \in  \mathcal{P}$ ensure the existence of a L\'evy process $\xi$  with ladder processes $H^+$ and $H^-$ and such that the Wiener-Hopf factorization \eqref{eq:wh} holds on $i\R$. Since we also assume that $k_+>0$, this L\'evy process necessarily drifts to $-\infty$.
Next let us introduce the Laplace exponents
\begin{eqnarray}\label{lL-KLadder}
 \phi^{(p)}_{+}(z)&=&
 \delta_{+}z + \int_{(0,\infty)}(\e^{zx}-1)\mu^{(p)}_+(\d x)-k_{+}\,,\\
 \phi_-^{(n)}(z)&=&
 -\delta_{-}z -\int_{(0,\infty)}(1-\e^{-zx})\mu^{(n)}_-(\d x),
 \end{eqnarray}
where we set $\mu_+^{(p)}(dx)=e^{-x/p}\mu_+(dx),\,p>0$, and $\mu_-^{(n)}(dx)=e^{-x/n}\mu_+(dx),\,n>0$. Plainly,  for any $p>0,\,n>0$,  $\mu_+^{(p)}\in \mathcal{P}$ and $\mu_-^{(n)} \in \mathcal{P}$, hence there exists a  L\'evy process $\xi^{(p,n)}$ with Laplace exponent $\Psi^{(p,n)}$ satisfying
 \begin{equation}
 \Psi^{(p,n)} (z) = -\phi^{(p)}_{+}(z)\phi^{(n)}_{-}(s),
 \end{equation}
which is easily seen to be analytic on the strip $-1/n <\Re(z)<1/p$. Moreover, from \cite[Corollary 4.4.4]{Doney}, we have $\E[\xi_1^{(p,n)}] = -k_+ \left(\int_0^{\infty}xe^{-x/n}\mu_+(dx)+\delta_-\right) $, which is clearly finite and negative.  Hence the conditions ${\bf E}_+$ are satisfied and we have, with the obvious notation, that
\[{\rm{I}}_{\xi^{(p,n)}}\stackrel{d}={\rm{I}}_{H_{(n)}^{-}}\times {\rm{I}}_{Y^{(p)}}\]
where for any $p>0$, $Y^{(p)}$ is a spectrally positive L\'evy process with Laplace exponent $\psi_+^{(p)}(-s)=-s\phi^{(p)}_{+}(-s),\: s\geq0$.
Let us first deal with the case $n\rightarrow \infty$. Since   $ \phi_-^{(n)}(s) \rightarrow  \phi_-(s)$, for all $s\geq0$, we have that
\[\lim_{n\rightarrow \infty}H_{(n)}^{-}\stackrel{d}=H^{-}\]
and  from Lemma \ref{Lemma1} (b) we get that
\[\lim_{n\rightarrow \infty}{\rm{I}}_{H_{(n)}^{-}}\stackrel{d}={\rm{I}}_{H^{-}}.\]
Thus, we deduce that, for any fixed $p>0$, the sequence  $({\rm{I}}_{\xi^{(p,n)}})_{n\geq1}$ is tight. Moreover,  for any fixed $p>0$,  we also have  $\xi^{(p,n)}\stackrel{d}{\rightarrow}\xi^{(p)}$, as $n\rightarrow \infty$, where $\xi^{(p)}$ has a Laplace exponent $\Psi^{(p)}$ given by
 \begin{equation}
 \Psi^{(p)}(z) = -\phi^{(p)}_{+}(z)\phi_{-}(z).
 \end{equation}
Indeed this is true by the philanthropy theory. Then from Lemma \ref{Lemma1} (c), we have that
\[\lim_{n\rightarrow \infty}{\rm{I}}_{\xi^{(p,n)}}\stackrel{d}={\rm{I}}_{\xi^{(p)}}\stackrel{d}={\rm{I}}_{H^{-}} \times {\rm{I}}_{Y^{(p)}},\]
which  provides a proof of the statement in the case ${\bf P}_{\pm}$ together with the existence of some finite positive exponential moments. Next,  as $p\rightarrow \infty,\: \phi_+^{(p)}(s) \rightarrow  \phi_+(s)$, for all $s\geq0$, and we have that
\[\lim_{p\rightarrow \infty}Y^{(p)}\stackrel{d}=Y,\]
where  $Y$ is a spectrally positive L\'evy process with Laplace exponent $\psi_+(-s)=-s\phi_{+}(-s)$.  As $\E[Y^{(p)}_1] = \phi_+^{(p)}(0)=-k_+ $, we can use Lemma \ref{Lemma1} (a) to get
\[\lim_{p\rightarrow \infty}{\rm{I}}_{Y^{(p)}}\stackrel{d}={\rm{I}}_{Y}.\]
As above, we conclude from Lemma \ref{Lemma1} (c) that
\[\lim_{p\rightarrow \infty}{\rm{I}}_{\xi^{(p)}}\stackrel{d}={\rm{I}}_{\xi}\stackrel{d}={\rm{I}}_{H^{-}}\times{\rm{I}}_{Y},\]
which completes the proof of the theorem. \QED
\section{Proof of the corollaries} \label{proof_cons}
\subsection{Corollary \ref{Corollary1}}
First, since $\xi$ is spectrally negative and has a negative mean, it is well known that the function $\Psi$ admits an analytical extension on the right-half plane which is convex on $\R^+$ drifting to $\infty$, with $\Psi'(0^+)<0$, and thus there exists $\gamma>0$ such that $\Psi(\gamma)=0$. Moreover, the Wiener-Hopf factorization for spectrally negative L\'evy processes boils down to
\[ \Psi(s)=\frac{\Psi(s)}{s-\gamma}(s-\gamma),\: s>0.\]
It is not difficult to check that with $\phi_+(s)=s-\gamma$ and $\phi_-(s)=-\frac{\Psi(s)}{s-\gamma}$, we have $\mu_-,\,\mu_+  \in \mathcal{P}$. Observing  that $\psi_+(s)=s^2-\gamma s$ is the Laplace exponent of a scaled Brownian motion with a negative drift $\gamma$, it is well-known, see e.g. \cite{Yor-01}, that
\[{\rm{I}}_Y \stackrel{d}{=}G_{\gamma}^{-1}.\]
The factorization follows then from Theorem \ref{MainTheorem} considered under the condition ${\bf P}_{\pm}$.  Since the random variable $G_{\gamma}^{-1}$ is  MSU, see \cite{Cuculescu}, we have that if ${\rm{I}}_{H^-}$ is unimodal then ${\rm{I}}_{\xi}$ is unimodal, which completes the proof of (1).
Next, (2) follows easily from the identity
\begin{eqnarray}
m_{{\xi}}(x)&=&\frac{1}{\Gamma(\gamma)}x^{-\gamma-1}\int_0^{\infty}e^{-y/x}y^{\gamma}m_{H^-}(y)dy \label{eq:dsn}
\end{eqnarray}
 combined with an argument of monotone convergence.

Further,  we recall  that Chazal et al.~\cite[Theorem 4.1]{Chazal-al-10} showed, that for any $\beta \geq 0$, $\phi_{\beta}(s) = \frac{s}{s+\beta}
\phi_-(s+\beta)$ is also the Laplace exponent of a negative of a subordinator and with the obvious notation
\begin{eqnarray} \label{eq:tbs}
m_{H^-_{\beta}}(x) &=& \frac{x^{\beta}m_{H^-}(x)}{\E[{\rm{I}}_{H^-}^\beta]},\quad x>0.
\end{eqnarray}
Then, assuming that  $1/x<\lim_{u\rightarrow \infty} \Psi(u)/u$, we have, from \eqref{eq:ms}, \eqref{eq:dsn} and \eqref{eq:tbs},
\begin{eqnarray*}
m_{{\xi}}(x)&=&\frac{1}{\Gamma(\gamma)}x^{-\gamma-1}\sum_{n=0}^{\infty}(-1)^n  \frac{x^{-n}}{n!}\int_0^{\infty}y^{n+\gamma}m_{H^-}(y)dy \\
&=&\frac{\E[{\rm{I}}_{H^-}^{\gamma}]}{\Gamma(\gamma)}x^{-\gamma-1}\sum_{n=0}^{\infty}(-1)^n \frac{x^{-n}}{n!}\frac{n!}{\prod_{k=1}^{n}-\frac{k}{k+\gamma}\phi_-(k+\gamma)} \\
&=&\frac{\E[{\rm{I}}_{H^-}^{\gamma}]}{\Gamma(\gamma)\Gamma(\gamma+1)}x^{-\gamma-1}\sum_{n=0}^{\infty}(-1)^n \frac{\Gamma(n+\gamma+1)}{\prod_{k=1}^{n}-k\phi_-(k+\gamma)} x^{-n}\\
&=&\frac{\E[{\rm{I}}_{H^-}^{\gamma}]}{\Gamma(\gamma)\Gamma(\gamma+1)}x^{-\gamma-1}\sum_{n=0}^{\infty}(-1)^n \frac{\Gamma(n+\gamma+1)}{\prod_{k=1}^{n}\Psi(k+\gamma)}x^{-n},
\end{eqnarray*}
where we used an argument of dominated convergence and the identity $-k\phi_-(k+\gamma)=\Psi(k+\gamma)$. Next, again from \eqref{eq:dsn}, we deduce that
\begin{eqnarray*}
x^{-\beta}m_{\xi}(x^{-1})&=&\frac{1}{\Gamma(\gamma)}x^{\gamma+1-\beta}\int_0^{\infty}e^{-xy}y^{\gamma}m_{H^-}(y)dy
\end{eqnarray*}
from where we easily see that, for any $\beta\geq \gamma+1$, the mapping $x\mapsto x^{-\beta}m_{\xi}(x^{-1})$ is completely monotone as the product of two Laplace transforms of positive measures. The proof of the Corollary is completed by invoking \cite[Theorem 51.6]{Sato-99} and noting that $\rm{I}^{-1}_{\xi}$ has a density given by $x^{-2}m_{\xi}(x^{-1})$, i.e. with $\beta=2$.
\subsection{Corollary \ref{Corollary2}}
We first observe from the equation \eqref{eq:ami-inv} that, in this case,
\begin{eqnarray*}
\bar{\mu}_+(x)&=&c e^{-\lambda x} \int_0^{\infty}e^{-\lambda y}\mathcal{U}_-(dy)\\
&=&c_-e^{-\lambda x},
\end{eqnarray*}
where  the last identity follows from \cite{Doney} and we have set $c_-=\frac{c}{\phi_-( \lambda)}$. From \eqref{Phi}, we deduce that $Y$ is a spectrally positive L\'evy process with Laplace exponent given, for any $s<\lambda$, by
\begin{eqnarray*}
 \psi_+(s)&=&\delta_+ s^2-k_+s + c_- \frac{s^{2}}{\lambda-s}\\
&=&\frac{s}{\lambda-s}\left(-\delta_+ s^2-(\delta_+\lambda+k_+ + c_-)s  -k_+\lambda \right),
\end{eqnarray*}
where $\delta_+>0$ since $\sigma>0$, see \cite[Corollary 4.4.4]{Doney}. Thus, using the continuity and convexity of $\psi_+$  on $(-\infty, \lambda)$ and on $(\lambda, \infty)$, studying its asymptotic behavior on these intervals and the identity  $\psi_+'(0)=-k_+<0$, we easily show that the equation  $\psi_+(s)=0$ has 3 roots which are real, one is obviously  $0$ and the two others $\theta_1$ and $\theta_2$ are such that $0<\theta_1<\lambda<\theta_2$. Thus,
\begin{eqnarray*}
 \psi_+(-s)
&=&\frac{\delta_+  s}{\lambda+s}\left(s+\theta_1\right)\left(s+\theta_2\right),\: s>-\lambda\,.
\end{eqnarray*}
Next, from \eqref{eq:msn}, we have,  with $C=k_+\frac{\Gamma(\lambda+1)}{  \Gamma(\theta_1+1)\Gamma(\theta_2+1)}$ and for $m=2,\ldots,$ that
\[ \E[{\rm{I}}_{Y}^{-m}] = C \delta_+^{m-1}   \frac{ \Gamma(m+\theta_1) \Gamma(m+\theta_2)}{\Gamma(m+\lambda)}\]
from where we easily deduce \eqref{eq:hy} by moments identification. Note that a simple computation gives that $\theta_1\theta_2=\delta_+\lambda k_+$ securing that the distribution of ${\rm{I}}_{Y}$ is proper. Next, the random variable ${\rm{I}}_{Y}^{-1}$ being moment determinate, we have, for $\Re(z)<\theta_1+1$, that
\[ \E[{\rm{I}}_{Y}^{z-1}] = C \delta_+^{-z}   \frac{ \Gamma(-z+\theta_1+1) \Gamma(-z+\theta_2+1)}{\Gamma(-z+\lambda+1)}.\]
 Applying the inverse Mellin transform, see e.g.~\cite[Section 3.4.2]{Paris}, we get
\begin{eqnarray}\label{Cor.2.4}
m_{Y}\left(\frac{x}{\delta^+}\right)&=& C \sum_{i=1}^2x^{-\theta_i-1} \mathcal{I}_i(- x^{-1}), \: x>0,
\end{eqnarray}
where $\mathcal{I}_i(x) = \sum_{n=0}^{\infty}b_{n,i}\frac{x^n}{n!}$, $b_{n,1}= \frac{\Gamma(\theta_2-\theta_1-n)}{\Gamma(\lambda-\theta_1-n)}$ and $ b_{n,2}=\frac{\Gamma(\theta_1-\theta_2-n)}{\Gamma(\lambda-\theta_2-n)}$.
The proof of the Corollary is  completed by following a line of reasoning similar to the proof of Corollary \ref{Corollary1}.

\subsection{Corollary \ref{Corollary3}}
For any $\alpha \in (0,1)$, let us observe that, for any $s \geq 0$,
\begin{eqnarray}
\phi_-(-s)&=& \frac{\alpha s \Gamma(\alpha (s+1)+1)}{(1+s)\Gamma(\alpha   s+1)}  \\
&=&\int_0^{\infty}(1-e^{-sy})u_{\alpha}(y)dy \label{eq:lpp},
\end{eqnarray}
 where $u_{\alpha}(y) =\frac{e^{-y}e^{-y/\alpha}}{\Gamma(1-\alpha)(1-e^{-y/\alpha})^{\alpha+1}}$.
We easily check that $u_{\alpha}(y)dy \in \mathcal{P}$ and hence $\Psi$ is a Laplace exponent of a L\'evy process which drifts to $-\infty$.  Next, we know, see e.g. \cite{Patie-aff}, that \[{\rm{I}}_{\tilde{H}^-}\stackrel{d}=S_\alpha^{-\alpha},\]
where $\tilde{H}^{-}$ is the negative of the subordinator having Laplace exponent \[\tilde{\phi}_-(-s) = \frac{\alpha\Gamma(\alpha s+1)}{\Gamma(\alpha(s-1)+1)}.\]
Observing that $\phi_-(-s) = \frac{-s}{-s+1} \tilde{\phi}_-(-s+1)$, we deduce, from \eqref{eq:tbs}, that
\begin{equation} \label{eq:ds}
m_{H^{-}}(x)=\frac{x^{-1/\alpha}}{\alpha}g_{\alpha}\left(x^{-1/\alpha}\right) ,\: x>0,\end{equation}
from which we readily get the expression \eqref{eq:dis}. Then,  we recall  the following power series representation of positive stable laws, see e.g. \cite[Formula (14.31)]{Sato-99},
\[ g_{\alpha}(x)= \sum_{n=1}^{\infty} \frac{(-1)^n}{\Gamma(-\alpha n) n!}x^{-(1+\alpha n)},\: x>0.\]
 Then, by means of an argument of dominated convergence justified by the condition $\lim_{s\rightarrow \infty}s^{\alpha-1}\phi_+(-s)=0$, we get, for all $x>0$, that
 \begin{eqnarray*}
m_{{\xi}}(x)
&=&\frac{k_+}{\alpha}\sum_{n=1}^{\infty} \frac{(-1)^n}{\Gamma(-\alpha n) n!}x^{n} \int_0^{\infty} y^{-(n+1)}f_{Y}(y)dy \\
&=&\frac{k_+}{\alpha}\sum_{n=1}^{\infty} \frac{\prod_{k=1}^{n}\phi_+(-k)}{\Gamma(-\alpha n) n!}x^{n},
\end{eqnarray*}
 where we used the identities \eqref{eq:msn},  $ \E[-Y_1]=k_+$ and $\psi_+(-k)=-k\phi_{+}(-k)$. The fact that the series is absolutely convergent is justified by using classical criteria combined with the Euler's reflection formula $\Gamma(1-z)\Gamma(z) \sin(\pi z)= \pi$ with the asymptotics
\begin{equation} \label{eq:ag}\frac{\Gamma(z+a)}{\Gamma(z+b)} = z^{a-b}\left(1+O\left(|z|^{-1}\right)\right) \quad \textrm{  as } z \to \infty, \:  |arg(z)|<\pi,\end{equation}
see e.g.~\cite[Chap.~1]{Lebedev-72}.
We complete the proof by mentioning  that Simon \cite{Simon} proved recently that  the positive stable laws are MSU if and only if $\alpha\leq 1/2$ which implies, from \eqref{eq:ds}, that ${\rm{I}}_{H^{-}}$ is also MSU in this case.
\providecommand{\bysame}{\leavevmode\hbox to3em{\hrulefill}\thinspace}
\providecommand{\MR}{\relax\ifhmode\unskip\space\fi MR }
\providecommand{\MRhref}[2]{%
  \href{http://www.ams.org/mathscinet-getitem?mr=#1}{#2}
}
\providecommand{\href}[2]{#2}

\end{document}